\documentclass[a4paper,12pt]{amsart}

\usepackage{a4wide}
\usepackage{graphicx,amssymb}
\usepackage{tikz}
\usepackage{enumitem}

\newtheorem{lemma}{Lemma}

\newtheorem{theorem}[lemma]{Theorem}
\newtheorem{corollary}[lemma]{Corollary}

\theoremstyle{remark}

\newtheorem{conjecture}{Conjecture}

\newtheorem{definition}{Definition}

\newcommand{\diam}{\operatorname{diam}}

\newcommand{\old}[1]{{}}

\usepackage{tikz}
\usetikzlibrary{automata}
\usetikzlibrary{shapes}
\usetikzlibrary{decorations.pathreplacing}
\usetikzlibrary{patterns,arrows}
\usepackage{hyperref}
\usepackage{mathtools}
\DeclarePairedDelimiter\abs{\lvert}{\rvert}

\DeclarePairedDelimiter\ceil{\lceil}{\rceil}
\DeclarePairedDelimiter\floor{\lfloor}{\rfloor}

\title{On the maximum diameter of $k$-colorable graphs}

\author{\'Eva Czabarka}
\author{Inne Singgih}
\author{L\'aszl\'o Sz\'ekely}
\address{\'Eva Czabarka\\Department of Mathematics \\ University of South Carolina \\ Columbia SC 29212 \\ USA
\and Visiting Professor\\ Department of Pure and Applied Mathematics\\ University of Johannesburg\\
P.O. Box 524, Auckland Park, Johannesburg 2006\\South Africa}
\email{czabarka@math.sc.edu}
\address{Inne Singgih\\Department of Mathematical Sciences\\ University of Cincinnati \\ Cincinnati, OH 45221 \\ USA}
\email{inne.singgih@uc.edu}
\address{L\'aszl\'o Sz\'ekely\\Department of Mathematics \\ University of South Carolina \\ Columbia SC 29212 \\ USA
\and Visiting Professor\\ Department of Pure and Applied Mathematics\\ University of Johannesburg\\
P.O. Box 524, Auckland Park, Johannesburg 2006\\ South Africa}
\email{szekely@math.sc.edu}

\thanks{The last two authors were  supported in part by the National Science Foundation  contract  DMS  1600811.}

\begin{document}
\begin{abstract} 
Erd\H{o}s, Pach, Pollack and Tuza   [\emph{J. Combin. Theory} {\bf B 47} (1989), 279--285] conjectured that  the diameter of a $K_{2r}$-free connected graph of order $n$ and 
minimum degree $\delta\geq 2$ is at most  $\frac{2(r-1)(3r+2)}{(2r^2-1)}\cdot \frac{n}{\delta} + O(1)$ for every $r\ge 2$, if
$\delta$ is a multiple of $(r-1)(3r+2)$. For every $r>1$ and  $\delta\ge 2(r-1)$, we create $K_{2r}$-free graphs with minimum degree $\delta$ and 
diameter $\frac{(6r-5)n}{(2r-1)\delta+2r-3}+O(1)$, which are counterexamples to the
conjecture for every $r>1$ and $\delta>2(r-1)(3r+2)(2r-3)$.  The rest of the paper proves positive results under a stronger hypothesis, $k$-colorability, instead of 
being $K_{k+1}$-free. We show that the diameter of connected $k$-colorable graphs with minimum degree
$\geq \delta$ and order $n$ is at most  $\left(3-\frac{1}{k-1}\right)\frac{n}{\delta}+O(1)$, while for 
$k=3$, it is
 at most $\frac{57n}{23\delta}+O\left(1\right)$.  
\end{abstract}
\maketitle

\section{Introduction}

The following theorem was
discovered several times \cite{Amar, EPPT, Gold, Moon}:

\begin{theorem}\label{th:ori}
For a fixed minimum degree $\delta \geq 2$ and  $n\rightarrow\infty$,  for every $n$-vertex connected graph $G$, we have  
$\diam(G) \leq \frac{3n}{\delta+1}+O(1)$. 
\end{theorem} 
Note that the upper bound is sharp (even for $\delta$-regular graphs \cite{Smyth}), but the constructions have complete subgraphs whose order increases with $\delta$.  
Erd\H{o}s, Pach, Pollack, and Tuza  \cite{EPPT} conjectured that the  upper bound in Theorem~\ref{th:ori} can be strengthened
for graphs not containing complete subgraphs:

\begin{conjecture}    \cite{EPPT}
\label{con:Erdosetal}
Let $r,\delta\geq 2$ be fixed integers and let $G$ be a connected graph of order $n$ and minimum degree $\delta$. 
\begin{enumerate}[label={\upshape (\roman*)}]
\item\label{conpart:even} If $G$ is $K_{2r}$-free and $\delta$ is a multiple of
$(r-1)(3r+2)$ then, as $n\rightarrow \infty$,
\[ \diam(G) \leq \frac{2(r-1)(3r+2)}{(2r^2-1)}\cdot \frac{n}{\delta} + O(1)=\left(3-\frac{2}{2r-1}-\frac{1}{(2r-1)(2r^2-1)}\right)\frac{n}{\delta}+O(1). \]
\item\label{conpart:odd}  If $G$ is $K_{2r+1}$-free and $\delta$ is a multiple of $3r-1$, then, as $n\rightarrow \infty$,
\[ \diam(G) \leq \frac{3r-1}{r}\cdot \frac{n}{\delta} + O(1)=\left(3-\frac{2}{2r}\right)\frac{n}{\delta}+O(1). \]
\end{enumerate}
\end{conjecture} 
Set $k=2r$ or $k=2r+1$ according the cases. As connected $\delta$-regular graphs are $K_{\delta+1}$-free (apart from $K_{\delta+1}$ itself), 
we need $\delta\ge k$ (at least) to make improvement on Theorem~\ref{th:ori}. Furthermore, as the conjectured constants in the bounds are at most $3-\frac{2}{k}$, 
Theorem~\ref{th:ori} implies that the conjectured inequalities hold  trivially, unless $\delta\geq\frac{3k}{2}-1$.

Erd\H{o}s \emph{et al.} \cite{EPPT} constructed graph sequences for every  $r,\delta\geq 2$, where   $\delta$ satisfies the divisibility condition, which 
meet the  upper bounds in Conjecture \ref{con:Erdosetal}. 
We show these construction them in Section~\ref{sec:clumpdef}.

Part (ii) of Conjecture~\ref{con:Erdosetal}  for $r=1$ was proved in 
Erd\H{o}s \emph{et al.} \cite{EPPT}.
Conjecture \ref{con:Erdosetal} is included in the book of Fan Chung and Ron Graham \cite{FanChung}, which collected Erd\H os's significant problems in graph theory.

No more progress has been reported on this conjecture, except that
for $r=2$ in (ii), under a stronger hypothesis  ($4$-colorable instead of $K_5$-free), Czabarka, Dankelman and Sz\'ekely \cite{dankelmanos} arrived at the
conclusion of Conjecture \ref{con:Erdosetal}:
\begin{theorem} 
\label{th:CDS}
For every connected $4$-colorable graph $G$ of order $n$ and
minimum degree $\delta\ge 1$,
\( \diam(G) \leq \frac{5n}{2\delta}-1. \)
\end{theorem}

In Section \ref{sec:counterex}, we give an unexpected counterexample for  Conjecture \ref{con:Erdosetal}~\ref{conpart:even} for
every  $r\geq 2$ and $\delta> 2 (r-1)(3r+2)(2r-3)$.  The question whether Conjecture \ref{con:Erdosetal}~\ref{conpart:even} holds in the 
range $(r-1)(3r+2)\le\delta\le 2(r-1)(3r+2)(2r-3)$ is still open.
The counterexample leads to a modification of  Conjecture \ref{con:Erdosetal}, which no longer requires cases: 
 \begin{conjecture} \label{con:7o3}
 For every $k\ge 3$ and $\delta\ge \lceil\frac{3k}{2}\rceil-1$, 
if $G$ is a $K_{k+1}$-free (weaker version: $k$-colorable) connected graph of order $n$ and  minimum degree at least $\delta$,  
  $\diam(G)\leq \left(3-\frac{2}{k}\right)\frac{n}{\delta}+O(1)$. 
    \end{conjecture}
For $k=2r$, Conjecture~\ref{con:7o3} is identical to Conjecture~\ref{con:Erdosetal}\ref{conpart:odd}.
For $k=2r-1$, $3-\frac{2}{k}=\frac{6r-5}{2r-1}$, and, although the conjectured bound is likely not tight for \emph{any} $\delta$,  the fraction
$\frac{6r-5}{2r-1}$ 
cannot be reduced for \emph{all} $\delta$ according to the construction in Section~\ref{sec:counterex}.

For the rest of the paper, we follow the restrictive approach of Czabarka, Dankelman and Sz\'ekely \cite{dankelmanos}, and work towards the weaker version
of Conjecture~\ref{con:7o3}.
 In other words, we use a stronger hypothesis ($k$-colorable instead of $K_{k+1}$-free) than what
Erd\H{o}s, Pach, Pollack, and Tuza  \cite{EPPT} used. In our work towards upper bounds on the diameter, we only assume \emph{minimum degree at least $\delta$}, 
a weaker assumption than \emph{minimum degree  $\delta$}.
Section \ref{sec:clump} shows that \emph{some} $k$-colorable (in particular $3$-colorable) connected graphs realizing the maximum diameter among such graphs
 with given order
and minimum degree have some canonical properties. Hence at proving upper bounds on the diameter, we can assume those canonical properties.

Section \ref{sec:dual} gives a linear programming duality approach to the maximum diameter problem. With this approach, proving upper bounds to the diameter
boils down to solve a  packing problem   in a graph, such that a certain value is reached by the objective function. If a packing with that value is given, the task of checking whether
the  packing is feasible is trivial. Using this approach we obtain
\begin{theorem}\label{th:upperbound} Assume $k\ge 3$. If $G$ is a connected $k$-colorable graph of minimum degree at least $\delta$, then
\[ \diam(G)\leq \frac{3k-4}{k-1}\cdot\frac{n}{\delta}+O(1)=\left(3-\frac{1}{k-1}\right)\frac{n}{\delta}+O(1).\]
\end{theorem}
This corroborates the conjecture of Erd\H{o}s \emph{et al.} in the  sense that the maximum diameter among all graphs investigated in Theorem~\ref{th:upperbound}  
is $\left(3-\Theta\bigl(\frac{1}{k}\right)\bigl)\frac{n}{\delta}$.
As a corollary, we arrive at 
the conclusion of
Theorem \ref{th:CDS}, if the graph is \emph{3-colorable} (instead of \emph{4-colorable}).  

Section \ref{sec:inex} 
applies  the inclusion-exclusion (sieve) formula   to give upper bounds \emph{locally} for the number of vertices in graphs with the \emph{canonical} properties.
In  Section \ref{sec:opt}, we define a number of global variables that play a role
in the  diameter problem, and turn the upper bounds from  Section \ref{sec:inex}  into
 linear constraints for the global variables.  (This approach was motivated by the flag algebra method of Razborov \cite{Razborov}.)
 A linear program of fixed size for the global variables arises, and solving this linear program proves our main positive result:
 \begin{theorem} \label{posmain}
For every connected $3$-colorable graph $G$ of order $n$ and minimum degree at least  
$\delta\ge 1$,
$$ \diam(G) \leq \frac{57n}{23\delta}+O(1). $$
\label{th:189o76}
\end{theorem}
 Note that as $57/23\approx  2.47826...$,   this is an improvement on 
 the  $\frac{5}{2}\cdot\frac{n}{\delta}+O(1)$ upper bound   for 4-colorable graphs  (see Theorem \ref{th:CDS}  cited from \cite{dankelmanos}). 
 In Theorem~\ref{special}, in a \emph{restricted  case} we prove  the weaker version of Conjecture~\ref{con:7o3} for $k=3$. 
 
The first and third authors thank Peter Dankelmann for introducing them to the problem and for suggesting the approach of using $k$-colorability 
instead of forbidden cliques.

\section{Clump Graphs and the Constructions for Conjecture~\ref{con:Erdosetal}}
\label{sec:clumpdef}

Let us be given a $k$-colorable connected graph $G$ of order $n$ and minimum degree at least $\delta$. Let  the eccentricity of vertex $x$ realize the  eccentricity
of the graph $G$,  $\diam(G)$.

Take  a \emph{fixed} good $k$-coloring of $G$. Let \emph{layer} $L_i$ denote the set of vertices at distance $i$ from $x$, and a \emph{clump} in $L_i$ be the set of vertices in $L_i$ that have the same color. The number of layers is  $\diam(G)+1$.

 Let $c(i) \in \{1,2,\ldots, k\}$ denote the number of colors used in layer $L_i$ by our fixed coloration. We can assume without loss of generality that
in $G$, two vertices in layer $L_i$, which are differently colored, are joined by an edge in $G$, and also that two vertices in consecutive layers, which are differently colored, are also joined by an edge in $G$. We call this assumption \emph{saturation}. 
Assuming saturation does not make loss of generality, as  adding these edges does not decrease degrees, keeps the fixed good $k$-coloration, and does not reduce
the diameter, while making the graph more structured for our convenience.

From a graph $G$ above, 
we create a 
\emph{(weighted) clump graph} $H$.  Vertices of $H$ correspond to the clumps of $G$. Two vertices of $H$ are connected by an edge if there were edges between the corresponding clumps in $G$. $H$ is naturally $k$-colored and layered based on the coloration and layering of $G$.  With a slight abuse of notation, we denote the layers of $H$ by
$L_i$ as well. 
We assign as \emph{weights} to each vertex of the clump graph  the number of vertices in the corresponding clump in $G$.

Given a (natural number)-weighted graph $H$, it defines a graph $G$ whose weighted clump graph is $H$  by blowing up vertices of $H$ into as many copies as their weight is. The degrees in $G$ correspond to the sum of the weights of neighbors of the vertices in $H$,  $\diam(G)=\diam(H)$, and the number of vertices in $G$ is the sum of the weights of all vertices in $H$.


It is convenient to describe the constructions of Erd\H{o}s \emph{et al.} \cite{EPPT} in terms of clump graphs. Any two consecutive layers of the clump graphs
will form a complete graph, and, as the order of these complete graphs will be at most $2r-1$ (resp. $2r$), the graphs will be $(2r-1)$-colorable and $K_{2r}$-free
(resp. $2r$-colorable and $K_{2r+1}$-free).

For the construction for $K_{2r}$-free graphs, when
$\delta$ is a multiple of $(r-1)(3r+2)$: 
Layer $L_0$ and has one clump, 
for $1\leq i \leq D$, for every odd $i$, layer  $L_i$ has $r$ clumps,
and for every even $i$, layer  $L_i$ has $r-1$ clumps
The clump in $L_0$ gets weight 1, and 
for $3\leq i \leq D-1$, for every odd $i$, the clumps in layer $L_i$ get and for 
$2\leq i \leq D-1$, for every even $i$, the clumps in layer $L_i$ get 
of weight  $\frac{(r+1)\delta}{(r-1)(3r+2)}$. Use weight $\delta$ for clumps in $L_1$ and 
$L_D$. (In case of $r=1$ and even $D$, use weight $\delta $ in the clumps of $L_{D-1}$ as well.)

For the construction for $K_{2r+1}$-free graphs, when
$\delta$ is a multiple of $3r-1$: 
Layer $L_0$ has one clump,
$1\leq i \leq D$, 
layer  $L_i$ has $r$ clumps.
The clump in $L_0$ gets weight 1, and  clumps in layers $L_i$ for $2\leq i \leq D-1$ get weight $\frac{\delta}{3r-1}$. Use weight $\delta$ for clumps in $L_1$ and 
$L_D$. 

The diameters of these constructions obviously meet the upper bounds of  Conjecture \ref{con:Erdosetal} within a constant term that depends on $r$.


\section{Counterexamples}
\label{sec:counterex}

We will make use of a clump graph to create a $(2r-1)$-colorable (and hence $K_{2r}$-free) graphs with minimum degree $\delta$ for every $r\ge 2$ that refute Conjecture~\ref{con:Erdosetal}~\ref{conpart:even}.

To make our quantities slightly more palatable in the description, we make the shift $s=r-1$, and work with $(2s+1)$-colorable graphs for $s\ge 1$.

For positive integers $p,s$ and $\delta\ge 2s$, we will create a weighted clump graph $H_{s,\delta,p}$ with $p(6s+1)$ layers, such that the number of vertices in two consecutive layers is at most $2s+1$, each vertex is adjacent to 
all other vertices in its own layer
and in the layers immediately before and after it. The layer structure of $H_{s,\delta,p}$ is basically \emph{periodic}, up to a tiny \emph{modification} in the weights. We are going to define a symmetric 
\emph{block} $C_{s,\delta}$ of $6s+1$ layers, and $H_{s,\delta,p}$ is the juxtaposition of $p$ copies of $C_{s,\delta}$, with the modification of increasing by 1 the weight 
of one vertex in the second layer $L_1$ and one vertex in the next-to-last layer $L_{p(6s+1)-1}$.

\begin{figure}[ht]
\centering
\begin{tikzpicture}[-,auto,node distance=2.5cm, semithick]
\tikzstyle{every state}=[fill=none,draw=black,text=black,scale=0.6]
  \node[state] (x0) [draw=white] { };
  \node[state] (1)  [right of=x0] [label=above:$1$] {$Z$};
  \node[state] (21) [above right of=1] [label=above:$\lfloor\frac{\delta-1}{2}\rfloor$] {$X$};
  \node[state] (22) [below right of=1] [label=below:$\lfloor\frac{\delta}{2}\rfloor$] {$Y$};
  \node[state] (3) [below right of=21] [label=above:$\lfloor\frac{\delta}{2}\rfloor$] {$Z$};
  \node[state] (4) [right of=3]  [label=below:$1$] {$Y$};
  \node[state] (5) [right of=4]  [label=above:$\lceil\frac{\delta}{2}\rceil$]  {$X$};
  \node[state] (61) [above right of=5] [label=above:$\lfloor\frac{\delta-1}{2}\rfloor$] {$Y$};
  \node[state] (62) [below right of=5] [label=below:$\lfloor\frac{\delta}{2}\rfloor$] {$Z$};
  \node[state] (7) [below right of=61] [label=above:$1$] {$X$};
  \node[state] (x1) [right of=7] [draw=white] { };

  \path (1)  edge node { } (21) edge node { } (22)
             edge [dashed] node { } (x0)
        (3)  edge node { } (21) edge node { } (22) edge node { } (4)
        (5)  edge node { } (4)  edge node { } (61) edge node { } (62)
        (7)  edge node { } (61) edge node { } (62) 
             edge [dashed] node { } (x1)
        (21) edge node { } (22)
        (61) edge node { } (62);
\end{tikzpicture}
\caption{The repetitive block $C_{1,\delta}$ for the weighted clump graph of the counterexample for $3$-colorable/$K_4$-free graphs. The letters $X,Y,Z$ give a $3$-coloration and the label above the vertex gives the weight of the vertex.} 
\label{fig:CountExOriginal}
\end{figure}
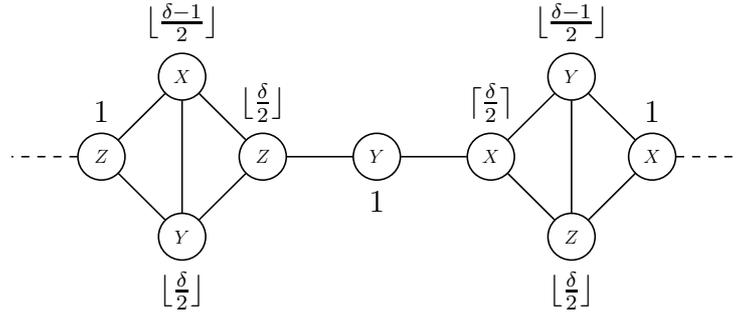
Let $0\leq d\leq 2s-1$ be the \emph{remainder}, when we divide $\delta$
with $2s$.
We  define $C_{s,\delta}$ by  the number of points and their weights in the layers $L_m$ for $0\leq m\leq 3s+1$ as detailed below; 
for $3s+2\leq m \leq 6s$, $L_m$ and the weights  will be the same as in $L_{6s-m}$.  In layers
$L_{3i\pm 1}$, every weight will be $\lfloor\frac{\delta}{2s}\rfloor$ or $\lceil\frac{\delta}{2s}\rceil$ before adjustment, and in layers $L_{3i}$ the weights will be 1. More precisely:
 \begin{enumerate}[label={\upshape (\Alph*)}]
\item For each $i:0\le i\le s$, let the layer $L_{3i}$ contain a single vertex with weight $1$.
\item For each $i:0\le i\le s-1$, let  the layer $L_{3i+1}$ contain $2s-i$ 
vertices, and assign them the following weights:
\begin{enumerate}[label={\upshape(\alph*)}]
\item If $d=0$,  let the weight of each of these vertices be $\frac{\delta}{2s}$. The adjustment is that   for a single  vertex in $L_1$, whose  weight is reduced to  $\frac{\delta}{2s}-1$. (By symmetry, the same adjustment happens in $L_{6s-1}$.)
\item If $d\geq 1$, then  let $\min(2s-i,d-1)$ vertices have weight $\lceil\frac{\delta}{2s}\rceil$, and the rest have weight $\lfloor\frac{\delta}{2}\rfloor$.
\end{enumerate}
\item For each $i:0\leq i< s-1$, let the layer $L_{3i+2}$ contain $i+1$ vertices, and assign them the following weights:
\begin{enumerate}[label={\upshape(\alph*)}]
\item If $d=0$, let the weight of each of them be $\frac{\delta}{2s}$.
\item If $1\leq d$, then let $d-\min(2s-i-1,d-1)$
vertices have weight $\lceil\frac{\delta}{2s}\rceil$, and the rest gets weight $\lfloor\frac{\delta}{2s}\rfloor$.
(This weight assignment is feasible. Note that $L_{3i+2}$ contains $i+1$ vertices, and, as  $d\le 2s-1$, $1\le d-\min(2s-i-1,d-1)\le i$). 
\end{enumerate}
\item Let layer $L_{3s-1}$ (and  symmetrically layer $L_{3s+1}$) have   $s$ vertices each. In these layers, let
$\lfloor\frac{d}{2}\rfloor$ vertices (resp.  $\lceil\frac{d}{2}\rceil$ vertices) have  weight $\lceil\frac{\delta}{2s}\rceil$,  and the remaining  vertices get weight $\lfloor\frac{\delta}{2s}\rfloor$. 
(This weight assignment is feasible. Since  $d\le 2s-1$, $\lceil\frac{d}{2}\rceil \le s$.)
\end{enumerate}
Note  that $\min(d-1,2s-i)=d-1$ for $i\in\{0,1,2\}$. We use this minimization for $i\leq s-2$. When
$s\leq 4$, we have $s-2\leq 2$, consequently there is no need to use the minimization formula for $s\leq 4$.
Therefore we show $C_{5,\delta}$ in Figure~\ref{fig:LargeCountEx}, which is the first instance to show all complexities of the counterexamples.  The case  $s=1$, when  $d\in\{0,1\}$, is even simpler: it is possible to describe  the weights without reference to $d$, see Figure~\ref{fig:CountExOriginal} for $C_{1,\delta}$.

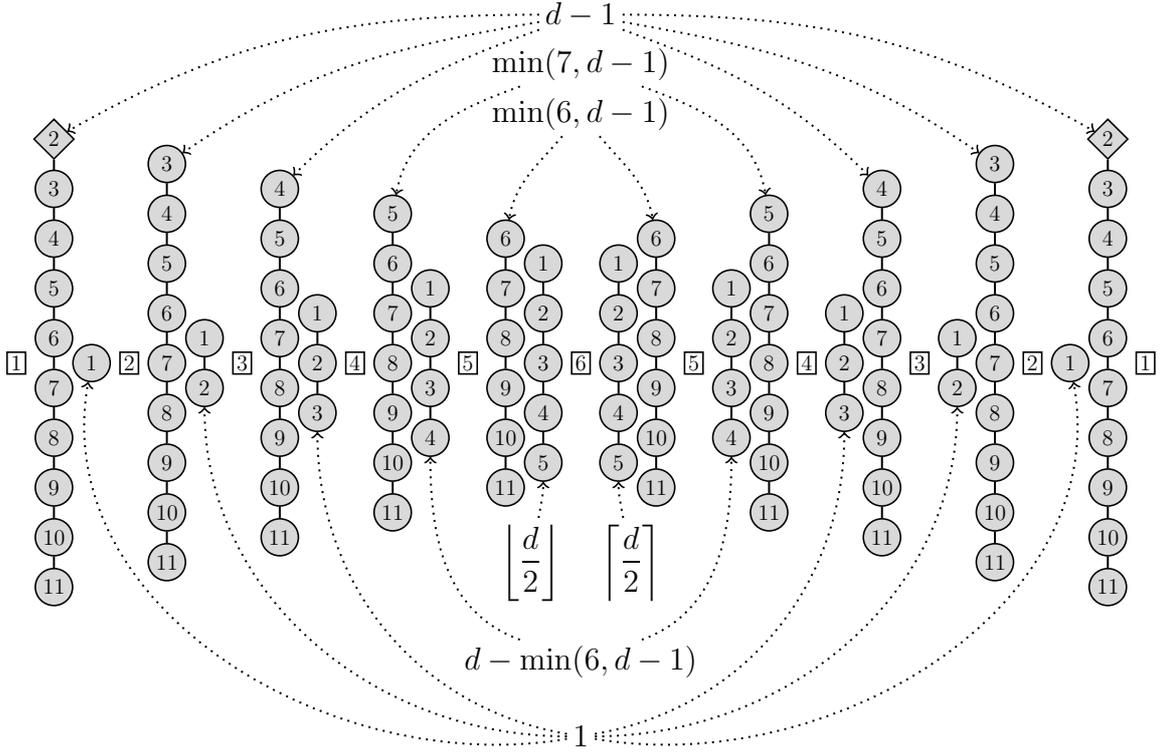
\begin{figure}[http]
		\centering 
			\begin{tikzpicture}
			[scale=.33,xscale=1,inner sep=2pt,semithick,
			vertexb/.style={circle,draw,fill=gray!30,minimum height=7mm, scale=.7},
			vertex/.style={rectangle,draw,minimum size=3mm,scale=.7},
			vertexc/.style={diamond,draw,fill=gray!30,minimum size=7mm,scale=.7},			
			thickedge/.style={line width=0.75pt},
			] 		
  \node[vertex] (x0) at (-24,0) [draw=white] { };
   \node[vertex] (0) at (-22.5,0) {$1$}; 
   \node[vertexc] (1a) at (-21,9) {$2$};
    \node[vertexb] (1b) at (-21,7) {$3$};    
    \node[vertexb] (1c) at (-21,5) {$4$};        
    \node[vertexb] (1d) at (-21,3) {$5$};    
     \node[vertexb] (1e) at (-21,1) {$6$};       
     \node[vertexb] (1f) at (-21,-1) {$7$}; 
     \node[vertexb] (1g) at (-21,-3) {$8$};         
     \node[vertexb] (1h) at (-21,-5) {$9$};     
     \node[vertexb] (1i) at (-21,-7) {$10$};         
     \node[vertexb] (1j) at (-21,-9) {$11$};                         
  \node[vertexb] (2) at (-19.5,0) {$1$};  
   \node[vertex] (3) at (-18,0) {$2$}; 
   \node[vertexb] (4a) at (-16.5,8) {$3$};
    \node[vertexb] (4b) at (-16.5,6) {$4$};    
    \node[vertexb] (4c) at (-16.5,4) {$5$};        
    \node[vertexb] (4d) at (-16.5,2) {$6$};    
    \node[vertexb] (4e) at (-16.5,0) {$7$};     
    \node[vertexb] (4f) at (-16.5,-2) {$8$};        
    \node[vertexb] (4g) at (-16.5,-4) {$9$};      
    \node[vertexb] (4h) at (-16.5,-6) {$10$};    
    \node[vertexb] (4i) at (-16.5,-8) {$11$};                    
  \node[vertexb] (5a) at (-15,1) {$1$};
  \node[vertexb] (5b) at (-15,-1) {$2$};
    \node[vertex] (6) at (-13.5,0) {$3$};
    \node[vertexb] (7a) at (-12,7) {$4$};
    \node[vertexb] (7b) at (-12,5) {$5$};    
    \node[vertexb] (7c) at (-12,3) {$6$};      
    \node[vertexb] (7d) at (-12,1) {$7$}; 
    \node[vertexb] (7e) at (-12,-1) {$8$};     
    \node[vertexb] (7f) at (-12,-3) {$9$};        
    \node[vertexb] (7g) at (-12,-5) {$10$};    
    \node[vertexb] (7h)at (-12,-7) {$11$};                                   
    \node[vertexb] (8a) at (-10.5,2) {$1$};
    \node[vertexb] (8b) at (-10.5,0) {$2$};    
    \node[vertexb] (8c) at (-10.5,-2) {$3$};        
  \node[vertex] (9) at (-9,0) {$4$};
    \node[vertexb] (10a) at (-7.5,6) {$5$};
    \node[vertexb] (10b) at (-7.5,4) {$6$};    
    \node[vertexb] (10c) at (-7.5,2) {$7$};      
    \node[vertexb] (10d) at (-7.5,0) {$8$};  
    \node[vertexb] (10e) at (-7.5,-2) {$9$};       
     \node[vertexb] (10f) at (-7.5,-4) {$10$};     
      \node[vertexb] (10g) at (-7.5,-6) {$11$};                        
    \node[vertexb] (11a) at (-6,3) {$1$};
    \node[vertexb] (11b) at (-6,1) {$2$};    
    \node[vertexb] (11c) at (-6,-1) {$3$};  
    \node[vertexb] (11d) at (-6,-3) {$4$}; 
  \node[vertex] (12) at (-4.5,0) {$5$};
    \node[vertexb] (13a) at (-3,5) {$6$};
    \node[vertexb] (13b) at (-3,3) {$7$};    
    \node[vertexb] (13c) at (-3,1) {$8$};      
    \node[vertexb] (13d) at (-3,-1) {$9$};  
    \node[vertexb] (13e) at (-3,-3) {$10$};    
    \node[vertexb] (13f) at (-3,-5) {$11$};                       
    \node[vertexb] (14a) at (-1.5,4) {$1$};
    \node[vertexb] (14b) at (-1.5,2) {$2$};    
    \node[vertexb] (14c) at (-1.5,0) {$3$};  
    \node[vertexb] (14d) at (-1.5,-2) {$4$}; 
    \node[vertexb] (14e) at (-1.5,-4) {$5$};     
   \node[vertex] (15) at (0,0) {$6$};			
  \node[vertex] (x1) at (24,0) [draw=white] { };
   \node[vertex] (30) at (22.5,0) {$1$}; 
   \node[vertexc] (29a) at (21,9) {$2$};
    \node[vertexb] (29b) at (21,7) {$3$};    
    \node[vertexb] (29c) at (21,5) {$4$};        
    \node[vertexb] (29d) at (21,3) {$5$};    
     \node[vertexb] (29e) at (21,1) {$6$};       
     \node[vertexb] (29f) at (21,-1) {$7$}; 
     \node[vertexb] (29g) at (21,-3) {$8$};         
     \node[vertexb] (29h) at (21,-5) {$9$};     
     \node[vertexb] (29i) at (21,-7) {$10$};         
     \node[vertexb] (29j) at (21,-9) {$11$};                         
  \node[vertexb] (28) at (19.5,0) {$1$};  
   \node[vertex] (27) at (18,0) {$2$}; 
   \node[vertexb] (26a) at (16.5,8) {$3$};
    \node[vertexb] (26b) at (16.5,6) {$4$};    
    \node[vertexb] (26c) at (16.5,4) {$5$};        
    \node[vertexb] (26d) at (16.5,2) {$6$};    
    \node[vertexb] (26e) at (16.5,0) {$7$};     
    \node[vertexb] (26f) at (16.5,-2) {$8$};        
    \node[vertexb] (26g) at (16.5,-4) {$9$};      
    \node[vertexb] (26h) at (16.5,-6) {$10$};    
    \node[vertexb] (26i) at (16.5,-8) {$11$};                    
  \node[vertexb] (25a) at (15,1) {$1$};
  \node[vertexb] (25b) at (15,-1) {$2$};
    \node[vertex] (24) at (13.5,0) {$3$};
    \node[vertexb] (23a) at (12,7) {$4$};
    \node[vertexb] (23b) at (12,5) {$5$};    
    \node[vertexb] (23c) at (12,3) {$6$};      
    \node[vertexb] (23d) at (12,1) {$7$}; 
    \node[vertexb] (23e) at (12,-1) {$8$};     
    \node[vertexb] (23f) at (12,-3) {$9$};        
    \node[vertexb] (23g) at (12,-5) {$10$};    
    \node[vertexb] (23h)at (12,-7) {$11$};                                   
    \node[vertexb] (22a) at (10.5,2) {$1$};
    \node[vertexb] (22b) at (10.5,0) {$2$};    
    \node[vertexb] (22c) at (10.5,-2) {$3$};        
  \node[vertex] (21) at (9,0) {$4$};
    \node[vertexb] (20a) at (7.5,6) {$5$};
    \node[vertexb] (20b) at (7.5,4) {$6$};    
    \node[vertexb] (20c) at (7.5,2) {$7$};      
    \node[vertexb] (20d) at (7.5,0) {$8$};  
    \node[vertexb] (20e) at (7.5,-2) {$9$};       
     \node[vertexb] (20f) at (7.5,-4) {$10$};     
      \node[vertexb] (20g) at (7.5,-6) {$11$};                        
    \node[vertexb] (19a) at (6,3) {$1$};
    \node[vertexb] (19b) at (6,1) {$2$};    
    \node[vertexb] (19c) at (6,-1) {$3$};  
    \node[vertexb] (19d) at (6,-3) {$4$}; 
  \node[vertex] (18) at (4.5,0) {$5$};
    \node[vertexb] (17a) at (3,5) {$6$};
    \node[vertexb] (17b) at (3,3) {$7$};    
    \node[vertexb] (17c) at (3,1) {$8$};      
    \node[vertexb] (17d) at (3,-1) {$9$};  
    \node[vertexb] (17e) at (3,-3) {$10$};    
    \node[vertexb] (17f) at (3,-5) {$11$};                       
    \node[vertexb] (16a) at (1.5,4) {$1$};
    \node[vertexb] (16b) at (1.5,2) {$2$};    
    \node[vertexb] (16c) at (1.5,0) {$3$};  
    \node[vertexb] (16d) at (1.5,-2) {$4$}; 
    \node[vertexb] (16e) at (1.5,-4) {$5$};

     \node(b1) at (0,14) {$d-1$};        
      \node(b4) at (0,10) {$\min(6,d-1)$};            
      \node(b7) at (0,12) {$\min(7,d-1)$};              
     \node(b2) at (0,-15) {$1$};                    
      \node(b8a) at (-2,-8) {$\displaystyle\left\lfloor\frac{d}{2}\right\rfloor$};          
       \node(b8b) at (2,-8) {$\displaystyle\left\lceil\frac{d}{2}\right\rceil$};                       
      \node(b10) at (0,-12) {$d-\min(6,d-1)$};            

\draw[thickedge] (1a)--(1b)--(1c)--(1d)--(1e)--(1f)--(1g)--(1h)--(1i)--(1j);
\draw[thickedge] (29a)--(29b)--(29c)--(29d)--(29e)--(29f)--(29g)--(29h)--(29i)--(29j);
\draw[thickedge] (4a)--(4b)--(4c)--(4d)--(4e)--(4f)--(4g)--(4h)--(4i);
\draw[thickedge] (26a)--(26b)--(26c)--(26d)--(26e)--(26f)--(26g)--(26h)--(26i);
\draw[thickedge] (5a)--(5b);
\draw[thickedge] (25a)--(25b);
\draw[thickedge] (7a)--(7b)--(7c)--(7d)--(7e)--(7f)--(7g)--(7h);
\draw[thickedge] (23a)--(23b)--(23c)--(23d)--(23e)--(23f)--(23g)--(23h);
\draw[thickedge] (8a)--(8b)--(8c);
\draw[thickedge] (22a)--(22b)--(22c);
\draw[thickedge] (10a)--(10b)--(10c)--(10d)--(10e)--(10f)--(10g);
\draw[thickedge] (20a)--(20b)--(20c)--(20d)--(20e)--(20f)--(20g);
\draw[thickedge] (11a)--(11b)--(11c)--(11d);
\draw[thickedge] (19a)--(19b)--(19c)--(19d);
\draw[thickedge] (13a)--(13b)--(13c)--(13d)--(13e)--(13f);
\draw[thickedge] (17a)--(17b)--(17c)--(17d)--(17e)--(17f);
\draw[thickedge] (14a)--(14b)--(14c)--(14d)--(14e);
\draw[thickedge] (16a)--(16b)--(16c)--(16d)--(16e);

\draw[->,thickedge,dotted] (b1) to [out=180,in=30] (1a);
\draw[->,thickedge,dotted] (b1) to [out=190, in=35] (4a);
\draw[->,thickedge,dotted] (b1) to [out=200, in=45] (7a);
\draw[->,thickedge,dotted] (b7) to [out=200, in=80] (10a);
\draw[->,thickedge,dotted] (b4) to [out=230, in=80] (13a);
\draw[->,thickedge,dotted] (b4) to [out=-50, in=100] (17a);
\draw[->,thickedge,dotted] (b7) to [out=-20, in=100] (20a);
\draw[->,thickedge,dotted] (b1) to [out=-20, in=135] (23a);
\draw[->,thickedge,dotted] (b1) to [out=-10, in=145] (26a);
\draw[->,thickedge,dotted] (b1) to [out=0, in=150] (29a);

\draw[->,thickedge,dotted] (b2) to [out=190,in=260] (2);
\draw[->,thickedge,dotted] (b2) to [out=180,in=270] (5b);
\draw[->,thickedge,dotted] (b2) to [out=170,in=270] (8c);
\draw[->,thickedge,dotted] (b10) to [out=160, in=270]   (11d);
\draw[->,thickedge,dotted] (b8a) to [out=80, in=270] (14e);
\draw[->,thickedge,dotted] (b8b) to [out=100, in=270] (16e);
\draw[->,thickedge,dotted] (b10) to [out=20, in=270]  (19d);
\draw[->,thickedge,dotted] (b2) to [out=10,in=270] (22c);
\draw[->,thickedge,dotted] (b2) to [out=0, in=270] (25b);
\draw[->,thickedge,dotted] (b2) to [out=-10,in=280] (28);
\end{tikzpicture}
\caption{The repetitive block $C_{5,\delta}$ of the weighted clump graph of the  for $11$-colorable/$K_{12}$-free  counterexample graphs. 
The vertices within a  layer are connected with a vertical line. Two vertices are connected, if they are in the same layer or in consecutive layers.
The numbers in the vertices give a good $11$-coloration. Before adjustment,
white rectangular vertices have weight $1$ and gray vertices have either weight $\lceil\frac{\delta}{10}\rceil$ or weight
$\lfloor\frac{\delta}{10}\rfloor$; and the numbers, from which  dotted arrows point to columns, give the number of vertices in the column that have weight 
$\lceil\frac{\delta}{10}\rceil$. Recall $d=\delta-10\lfloor\frac{\delta}{10}\rfloor$. 
The adjustment: if $d=0$, the weight of the two diamond shaped vertices are decreased by $1$.
} 
\label{fig:LargeCountEx}
\end{figure}

\begin{lemma}\label{lm:clumpH} Let $p\ge 1$ and $s\ge 2$. The weighted clump graph $H_{s,\delta,p}$ has the following properties:
\begin{enumerate}[label={\upshape (\alph*)}]
\item\label{part:Htriv} $H_{s,\delta,p}$ is $(2s+1)$-colorable with diameter $p(6s+1)-1$.
\item\label{part:Htotal} The sum of the weights of all  vertices is $p\bigl((2s+1)\delta+2s-1\bigl)+2$.
\item\label{part:Hmindeg} For any vertex $y\in V(H_{s,\delta,p})$, the sum of the weights of its neighbors is at least $\delta$.
\end{enumerate}
\end{lemma}  

\begin{proof}  ~\ref{part:Htriv} The statement on the diameter  is trivial.
As the number of vertices in any two consecutive layer of $H_{s,\delta,p}$ is at most $2s+1$, we can $(2s+1)$-color $H_{s,\delta,p}$ with $(2s+1)$ colors from left to right greedily. 

 \ref{part:Htotal} If $W$ is the sum the weights of vertices in the  block $C_{s,\delta}$, then the total sum of weights in $H_{s,\delta,p}$  is $pW+2$ (the 2 is due to the modification), so
we need to show that $W=(2s+1)\delta+2s-1$.

Consider an $i$ with $0<i<s-1$. $L_{3i-1}\cup L_{3i+1} $ has $(i-1)+1+2s-i=2s$ vertices. If $d=0$, each of them has weight $\frac{\delta}{2s}$,
otherwise $d-\min(d-1,2s-i)+\min(d-1,2s-i)=d$ of them have weight $\lceil\frac{\delta}{2s}\rceil$, and the rest have weight
$\floor{\frac{\delta}{2s}}$. 
So the sum of the weight of the vertices in $L_{3i-1}\cup L_{3i}\cup L_{3i+1} $ is  $\delta+1$, and so is in $L_{6s-3i+1}\cup L_{6s-3i}\cup L_{6s-3i-1}$.

$L_{3s-1}\cup L_{3s+1}$ contains $2s$ vertices, $\floor{\frac{d}{2}}+\ceil{\frac{d}{2}}=d$ of them has weight $\ceil{\frac{\delta}{2s}}$, the rest
$\floor{\frac{\delta}{2s}}$, so the sum of the weights of the vertices in
$L_{3s-1}\cup L_{3s}\cup L_{3s+1}$ is also $\delta+1$.

$L_1$ has $2s$ vertices. If $d=0$, one of these have weight $\frac{\delta}{2s}-1$  
and the rest have weight $\frac{\delta}{2s}$.
If $d>0$, $d-1$ of the vertices have weight $\lceil\frac{\delta}{2s}\rceil$, the rest has weight $\floor{\frac{\delta}{2s}}$.
 The sum of the weights 
in $L_0\cup L_1$ is $1+\delta-1=\delta$.

So
$W=2\delta+(2s-1)(\delta+1)=(2s+1)\delta+2s-1$, which finishes the proof of \ref{part:Htotal}.

For  \ref{part:Hmindeg}: Let $y$ be a vertex of $H_{s,\delta,p}$.
Then for some $j$ $(0\le j<p)$, $y$ is in the $j$-th block $C_{s,\delta}$,  and for some $m$ $(0\leq m \leq 6s)$, $y$ is in 
 the layer $L_m$ of $C_{s,\delta}$. Because of symmetry, we may assume that $0\leq m\leq 3s$. 
 The weights in the layer $L_{3s-1}$ are less or equal than the weights in the layer $L_{3s+1}$,
but may not be equal, breaking the symmetry, but still handling cases with $0\leq m  \leq 3s$ gives a $\delta$ lower bound to the degrees of all vertices of 
$H_{s,\delta,p}$. In addition, layer $(j,m)= (p-1,6s-1)$, where a modification happened,  is symmetric to the layer  $(j,m)=(0,1)$, where identical modification happened. 
Therefore checking the degrees of the vertices in the first half of the first (and modified) copy of  $C_{s,\delta}$  in $H_{s,\delta,p}$ covers checking the degrees in the
second half of the last (and modified) copy of  $C_{s,\delta}$  in $H_{s,\delta,p}$.

If  $y\in L_{3i}$ for some $0<i\le 2s$, then  $y$ has weight $1$ and is adjacent to all vertices but itself in $L_{3i-1}\cup L_{3i}\cup L_{3i+1} $.  As we have already shown  in  the proof of part \ref{part:Htotal},    $L_{3i-1}\cup L_{3i}\cup L_{3i+1} $  has total weight $\delta+1$, the neighbors of $y$ have total weight $\delta$.

If $y\in L_0$, then as we showed in    the proof of part \ref{part:Htotal},        the total weight of the vertices in $L_0$ in an unmodified block, which is not the first or the last block,  is $\delta-1$.
Either $y$ is adjacent  
to a vertex of weight $1$ outside of its own block, or $y$ is in a modified block where the total weight of $L_1$ got increased by $1$: in both cases
the sum of the weights of the neighbors of $y$ is $\delta$. 

Assume now that $y$ is a vertex of $L_{3i+1}\cup L_{3i+2}$ for some $0\le i\le s-1$.
 Note  $L_{3i+1}\cup L_{3i+2}$ contains $2s+1$ vertices, $2s$ of which is the neighbor of
$y$, plus $y$ has a neighbor of weight $1$ outside of  $L_{3i+1}\cup L_{3i+2}$. We consider two cases for $d$:

If $d=0$   and $0< i\le s-1$, then each neighbor  of $y$ in $L_{3i+1}\cup L_{3i+2}$ has weight $\frac{\delta}{2s}$, so the sum of the weights of the neighbors of $y$ is $\delta+1$.   If $d=0$   and $i=0$,   because of the adjustment,  the sum of the weights of the neighbors may  decrease by $1$, and is still  $\geq \delta$.

If $d>0$, then all vertices in $L_{3i+1}\cup L_{3i+2}$ have weight at least $\lfloor\frac{\delta}{2s}\rfloor$. 
If $i<s-1$, then  at least
$\min(d-1,2s-i)+d-\min(d-1,2s-i-1)\ge d$ many vertices of $L_{3i+1}\cup L_{3i+2}$ have weight $\lceil\frac{\delta}{2s}\rceil$. 
If $i=s-1$, then, as $s+1>\lceil\frac{d}{2}\rceil$, we have that  $\lfloor\frac{d}{2}\rfloor\ge\max(1,d-(s+1))$,
so  $L_{3s-2}\cup L_{3s-1}$ has at least $\min(d-1,s+1)+\floor{\frac{d}{2}}=d-\max(1,d-(s+1))+\floor{\frac{d}{2}}\ge d$ vertices with weight $\lceil\frac{\delta}{2s}\rceil$. 
Therefore, for any $0\leq i\leq s-1$, any $y\in   L_{3i+1}\cup L_{3i+2}$ has at least
$d-1$ neighbors of weight $\floor{\frac{\delta}{2s}}+1$, the total weight of $y$'s neighbors is at least
$2s\lfloor\frac{\delta}{2s}\rfloor+d-1+1=\delta$. This finishes the proof of ~\ref{part:Hmindeg}.
\end{proof}

\begin{theorem} Let $r\ge 2$, $\delta\ge 2r-2$, and for each positive integer $p$, let $G_{r,\delta,p}$ be the graph whose weighted clump graph is
$H_{r-1,\delta,p}$. Then $G_{r,\delta,p}$ is $2r-1$ colorable (and hence $K_{2r}$-free), connected, with minimum degree  $\delta$, 
of order $n=p\bigl((2r-1)\delta+2r-3\bigl)+2$,
and of diameter  $\frac{(6r-5)n}{(2r-1)\delta+2r-3}+O(1)$. Consequently, \emph{ Conjecture~\ref{con:Erdosetal}} fails for every $\delta> 12r^3-22r^2-2r+12=2(r-1)(3r+2)(2r-3)$. Furthermore, the difference between the coefficient of $\frac{n}{\delta} $ in our construction and in \emph{ Conjecture~\ref{con:Erdosetal}\ref{conpart:even}} is 
$\frac{1}{(2r^2-1)(2r-1)}+o(1)$, as $\delta\rightarrow\infty$.
\end{theorem}

\begin{proof} By Lemma~\ref{lm:clumpH}, $G_{r,\delta,p}$
is $(2r-1)$-colorable with minimum degree $\delta$, diameter $p(6r-5)-1$, and it has  
$n=p\bigl((2r-1)\delta+2r-3\bigl)+2$ vertices. Therefore its diameter is
$\frac{(6r-5)(n-2)}{(2r-1)\delta+2r-3}-1=\frac{(6r-5)n}{(2r-1)\delta+2r-3}+O(1)$.
Consider the identity
\[
\frac{(6r-5)\delta}{(2r-1)\delta+2r-3}-
\frac{2(r-1)(3r+2)}{(2r^2-1)}
=\frac{1}{(2r^2-1)(2r-1)}\cdot \frac{1-\frac{12r^3-22r^2-2r+12}{\delta}}{(1+\frac{2r-3}{(2r-1)\delta})}.
\]
This shows both the fact that
$\frac{(6r-5)n}{(2r-1)\delta+2r-3}
\leq \frac{2(r-1)(3r+2)}{(2r^2-1)}\cdot \frac{n}{\delta} 
$ iff $\delta\leq 12r^3-22r^2-2r+12$, and the statement about the difference.
  \end{proof}

\section{Canonical Clump Graphs}
\label{sec:clump}
We use the letters $X,Y,Z$ to denote three unspecified but different colors from our $k$ colors.

\begin{theorem}\label{th:canonical} Assume $k\ge 3$. Let $G'$ be a $k$-colorable connected  graph of order $n$, diameter $D$ and minimum degree at least $\delta$. 
Then there is a $k$-colored  connected graph $G$ of the same parameters, with layers $L_0,\ldots,L_D$, for which  the following hold for every
$i$ $(0\le i\le D-1)$:
\begin{enumerate}[label={\upshape (\roman*)}]
\item\label{part:1k} If  $c(i)=1$, then $c(i+1)\le k-1$.
\item\label{part:manycolor} The number of colors used to color the set $L_i\cup L_{i+1}$ is $\min(k,c(i)+c(i+1))$. In particular, when
$c(i)+c(i+1)\le k$, then $L_i$ and $L_{i+1}$ do not share any color.
\item\label{part:k1} If $c(i)=k$, then $i\ge 2$ and  $c(i+1)\ge 2$.
\item\label{part:singleton} If $|L_i|>c(i)$, i.e., $L_i$ contains two vertices of the same color, then $i>0$ and $c(i)+\max\bigl(c(i-1),c(i+1)\bigl)\geq k$.
\end{enumerate}
\end{theorem}

\begin{proof}
After having proved a part of the Theorem, we will assume that $G'$ itself satisfies that property when we  complete  the  proof of the remaining parts. When we create 
new $G'$ graphs, they will still satisfy the already checked parts, in other words, we do not regress to issues that we already resolved.
 We fix a $k$-coloration of $G'$,  let $x_0$ be  a vertex of eccentricity $D$ in $G'$, and let  $L_0,\ldots,L_D$ be the distance layering of $G'$. Without loss of generality, we  assume that $G'$ is saturated.

\ref{part:1k} Select $G=G'$ with the same $k$-coloration, The statement
 follows from the fact that every vertex in $L_{i+1}$ has a neighbor in $L_i$; therefore if color $X$ appears in $L_{i+1}$, then $L_i$ has at least one color different from $X$.

 If \ref{part:manycolor} or \ref{part:singleton} is not satisfied in $G'$, our general strategy is the following: create a new $k$-coloring of the vertices of $G'$ such that
the set of the vertices in any layer does not change, vertices of different color will remain differently colored, and in the new coloring 
the already proven statements  still hold. We saturate $G'$ in the new coloring (by adding new edges, if needed) to obtain $G$. Now we complete this
strategy for \ref{part:manycolor}, and postpone the proof of \ref{part:singleton}  till the end.

If \ref{part:manycolor} fails in $G'$, consider the smallest $i$, such that the set $L_i\cup L_{i+1}$ contains fewer than $\min(k,c(i)+c(i+1))$ colors. By
\ref{part:1k}, $i>0$. Observe that  
there are different colors
$X,Y$ such that $X$ is used in both of $L_i$ and $L_{i+1}$, while $Y$ is not used in $L_i\cup L_{i+1}$.
We define a new coloration by switching colors $X$ and $Y$ in all $L_j$ for all $ j\ge i+1$. This is a good coloration, in which $L_i\cup L_{i+1}$ uses one more color.
Repeated application of this procedure yields a $k$-coloration where \ref{part:manycolor} holds.

The hard part of this theorem is \ref{part:k1}.
If $c(i)=k$, then by \ref{part:1k} $i\ge 2$. If $c(i+1)=1$ (i.e., if \ref{part:k1} fails), then we will move clumps  within
$L_{i-1}\cup L_{i}$, and  recolor of the graph,
such that the resulting layered colored graph will have the same required parameters as $G'$, creating no violations of  \ref{part:1k}, \ref{part:manycolor},                     
 and reducing the number of violations of \ref{part:k1} in $G'$.

Let $X$ be the color used on $L_{i+1}$. 

Assume first that $L_{i-2}$ contains a color different from $X$. Set $S$ be the set of vertices in
$L_i$ that is colored $X$. Move the vertices of $S$ from $L_i$ to $L_{i-1}$ without recoloring them, 
either merging them into the $X$-colored clump of $L_{i-1}$ or creating one, if no such clump existed in $L_{i-1}$.
Add new edges to achieve saturation.  
 In the resulting graph,
the layer indexed by $i$ contains $k-1$ colors,   reducing the number of violations of \ref{part:k1} in $G'$, 
and not creating any violation of  \ref{part:1k} or \ref{part:manycolor}.

Hence in the following we may assume that $c(i-2)=1$ and $L_{i-2}$ is colored with color $X$. By \ref{part:1k}, we have $c(i-1)\le k-1$.
If $c(i-2)<k-2$, then there is a color $Y$ not used in $L_{i-2}\cup L_{i-1}$. Recolor $G'$ by switching colors $X$ and $Y$ in $L_j$ $(0\leq j\le i-2)$ and 
recover saturation. 
In the new coloring $L_{i-2}$ has a color different from $X$, and we are back to the  case we already handled above.

Hence in the rest we may assume that $c(i-2)=c(i+1)=1$, $c(i)=k$, $c(i-1)=k-1$, and both $L_{i-2}$ and $L_{i+1}$ are colored with $X$.
Let $Y,Z$ be two arbitrary colors different from $X$, and $S$ be the set of vertices in $L_{i-1}\cup L_i$ colored with $X$, $Y$ or $Z$. We will repartition and recolor (only with colors $X,Y,Z$) the vertices in $S$, and possibly recolor $L_{i+1}$ from color $X$ to color $Y$.  If we recolor $L_{i+1}$, then we exchange the colors $X$ and $Y$ in all layers  $L_j$ for $ j\ge i+2$. After these steps, we recover saturation  in $G'$. After the changes, in $G'$ both $L_{i-1}$ and $L_i$ will contain fewer than $k$ colors, and in the resulting $k$-colored graph the diameter, the order and minimum degree condition do not change, and no instances violating \ref{part:1k} and \ref{part:manycolor}
will be created, and we reduced the number of violations of     \ref{part:k1}.  
The difficulty is in maintaining the minimum degree condition in $G'$ along these operations. This is what we check next, and the  
repartitioning and recoloring  of the vertices in $S$ will depend on some inequalities between certain clump sizes.

If $y$ is a vertex not in $L_{i-2}\cup L_{i-1}\cup L_i\cup L_{i+1}$ or
$y$ is not colored with $X$, $Y$ or $Z$ in the  graph before the operations, then the neighborhood set of $y$ does not change. 
If $y$ is a vertex in $L_{i-2}\cup L_{i-1}\cup L_i\cup L_{i+1}$
colored with one of $X,Y,Z$, then the symmetric difference between the new and old neighborhood set of $Y$ is a subset of $S$. Therefore we only need to check
the minimum degree condition for  
 vertices colored $X,Y$ or $Z$ in $L_{i-2}\cup L_{i-1}\cup L_i\cup L_{i+1}$, and we have to show that after the operations 
 they have at least as many   $X,Y,Z$  colored neighbors in $L_{i-2}\cup L_{i-1}\cup L_i\cup L_{i+1}$ as before the operations.
   For any $j$ $(1\le j\le 4)$, we will denote by $x_j,y_j$ and $z_j$
the number of vertices in $L_{i+j-3}$ colored $X,Y$ and $Z$ in $G'$ respectively, \emph{before the operations}. 
The $k\geq 3$ assumption, together with the fact that
$L_{i-1}$ has no color $X$  by    \ref{part:manycolor},  implies that $x_1$, $y_2$, $z_2$, $y_3$, $z_3$ and $x_4$ are positive.

We have several cases to consider:
\begin{enumerate}[label={\upshape{\fbox{\arabic*}}}]
    \item\label{casebox1} $x_3\geq y_3$ or $x_3\geq z_3$.\\ 
       It suffices to handle the case $x_3\ge y_3$, as the case $x_3\ge z_3$ can be handled similarly.
   Let the operations create
    in $L_{i-1}$  a  clump of size $y_2+y_3$ of color $Y$, and a clump of size $z_2$ of color $Z$; in $L_i$
    a clump of size $x_3$ of color $X$ and a clump of size $z_3$ of color $Z$. Recolor  $L_{i+1}$ with $Y$, and switch colors $X$ and $Y$ in every
    $L_j$ for $ j\ge i+2$, 
   see Fig.~\ref{fig:box1}. 
Note that, as claimed, $|L_{i-1}\cup L_i|$ did not change. We verify the minimum degree condition.  Let $d(W_i)$ to denote the number of neighbors of a vertex $w$ from
the  clump colored $W$ in layer $L_i$ among the $X,Y,Z$ colored vertices of $L_{i-2}\cup L_{i-1}\cup L_i\cup L_{i+1}$
before the operations,  and $d'(W_i)$ to denote the degree of a vertex  $w'$ from the clump colored $W$ in layer $L_i$ 
among the $X,Y,Z$ colored vertices of $L_{i-2}\cup L_{i-1}\cup L_i\cup L_{i+1}$
after the operations. We have:
\begin{eqnarray*}
d'(X_{i-2})&=&d(X_{i-2})+y_3,\\
d'(Y_{i-1})&=&x_1+z_2+x_3+z_3                         =d(B_i),\\
d'(Z_{i-1})&=&x_1+y_2+z_3+y_3=d'(Z_{i-1}),\\
d'(X_{i})&=&y_2+y_3+z_2+z_3+x_4>y_2+y_3+z_2+z_3=d(X_{i}),\\
d'(Z_{i})&=&y_2+y_3+x_3+x_4  =    d(Z_{i})         ,\\    
d'(Y_{i+1})&=& x_3+z_3=d'(X_{i+1})+(x_3-y_3).
\end{eqnarray*}

\begin{figure}[ht]
    \centering
    \begin{tikzpicture}[-,auto,node distance=2.25cm, thick, scale=0.7]
    \tikzstyle{every state}=[fill=none,draw=black,text=black,scale=0.6]
    \node[state] (0)  [draw=white] at (-1.5,0){ };
    \node[state] (1)  [label={\small $x_1$}] at (0,0) {$X$};
    \node[state] (21) [label={\small $y_2$}] at (2,2)  {$Y$};
    \node[state] (22) [label=below:{\small $z_2$}] at (2,-2) {$Z$};
    \node[state] (32) [label={\small $y_3$}] at ( 4,2) {$Y$};
    \node[state] (31) [label={[xshift=2.5mm]\small $x_3$}] at (4,0) {$X$};
    \node[state] (33)  [label=below:{\small $z_3$}]  at (4,-2) {$Z$};
    \node[state] (4)  [label={\small $x_4$}] at (6,0) {$X$};
    \node[state] (5)  [draw=white] at (7.5,0) { };
   
    \node[state] (x1) [label={\small $x_1$}] at (9,0) {$X$};
    \node[state] (x21) [label={[label distance=1mm]\small $y_2+y_3$}] at (11,2) {$Y$};
    \node[state] (x22) [label=below:{\small $z_2$}] at (11,-2) {$Z$};
    \node[state] (x32)  [label={[xshift=0mm, label distance=-12mm]\small $z_3$}] at (13,-2) {$Z$};
        \node[state] (x31) [label={\small $x_3$}] at (13,0) {$X$};
    \node[state] (x4)  [label={\small $y_4$}] at (15,2) {$Y$};
    \node[state] (x5)  [draw=white] at (16.5,2) { };
   
    \path (0)  edge [dashed] node { }(1)
          (1)  edge node { }(21) edge node { }(22)
            (21) edge node { }(22) edge node { }(31) edge node { }(33)           
          (22) edge node { }(31) 
          (22) edge node { }(32) %
          (31) edge node { }(32) 
          (32) edge [bend right=25] node { }(33)
         (31) edge node { }(33)
          (4)  edge node { }(32) edge node { }(33)  edge [dashed] node { }(5)
                    
          (x1)  edge node { }(x21) edge node { }(x22)
         (x32) edge node {} (x21) 
         (x21) edge node { }(x22) 
         (x21) edge node { }(x31)
         (x4)  edge node { }(x32)
         (x4)  edge node { }(x31) edge [dashed] node { } (x5) 
          (x22) edge node { }(x31)
         (x31) edge node { }(x32)
          (x1)  edge [dashed] node { } (5)    ;
    \end{tikzpicture}
    \caption{When $x_3 \geq y_3$, before  and after the operations (left and right). }
    \label{fig:box1}
\end{figure}
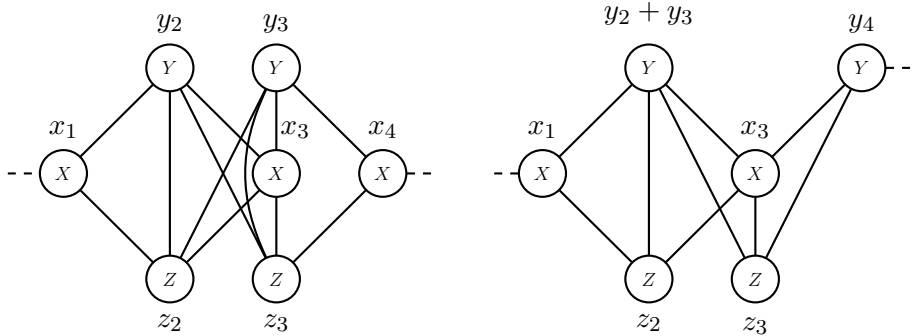

    \item $x_3<y_3$ and $x_3<z_3$ and  ($x_3\geq y_2$ or $x_3\geq z_2$).\\ 
        We may  assume $x_3\geq y_2$, as $x_3\geq z_2$ can be handled similarly.  Let the operations create
        in $L_{i-1}$  a clump of size $x_3$ of color $Y$ and a clump of size $z_2$ of color $Z$; 
        and in $L_{i}$ create a clump of size $y_3+y_2$ of color $X$ and a clump $z_3$ of color $Z$; recolor
        $L_{i+1}$ to color $Y$ and switch colors $X$ and $Y$ in $L_j$ for $j\ge i+2$, see Figure \ref{fig:box2}.
\vspace{-2mm}        
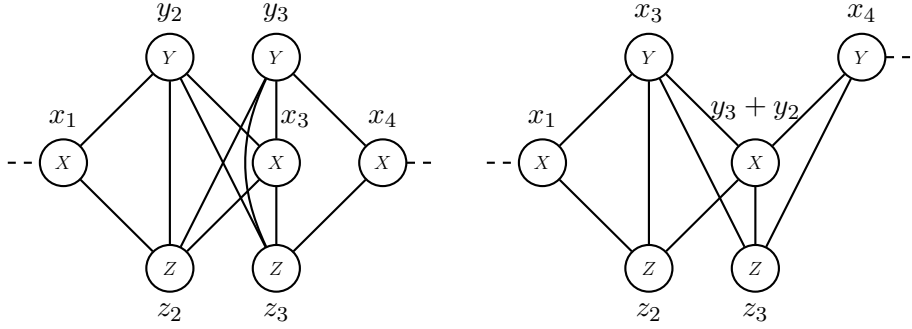
\begin{figure}[ht]
    \centering
    \begin{tikzpicture}[-,auto,node distance=2.25cm, thick, scale=0.7]
    \tikzstyle{every state}=[fill=none,draw=black,text=black,scale=0.6]
    \node[state] (0)  [draw=white] at (-1.5,0){ };
    \node[state] (1)  [label={\small $x_1$}] at (0,0) {$X$};
    \node[state] (21) [label={\small $y_2$}] at (2,2)  {$Y$};
    \node[state] (22) [label=below:{\small $z_2$}] at (2,-2) {$Z$};
    \node[state] (32) [label={\small $y_3$}] at ( 4,2) {$Y$};
    \node[state] (31) [label={[xshift=2.5mm]\small $x_3$}] at (4,0) {$X$};
    \node[state] (33)  [label=below:{\small $z_3$}]  at (4,-2) {$Z$};
    \node[state] (4)  [label={\small $x_4$}] at (6,0) {$X$};
    \node[state] (5)  [draw=white] at (7.5,0) { };
   
    \node[state] (x1) [label={\small $x_1$}] at (9,0) {$X$};
    \node[state] (x21) [label={\small $x_3$}] at (11,2) {$Y$};
    \node[state] (x22) [label=below:{\small $z_2$}] at (11,-2) {$Z$};
    \node[state] (x32) [label={[label distance=1mm]\small $y_3+y_2$}] at (13,0) {$X$};
    \node[state] (x33) [label=below:{\small $z_3$}] at (13, -2) {$Z$};
    \node[state] (x4)  [label={\small $x_4$}] at (15,2) {$Y$};
    \node[state] (x5)  [draw=white] at (16.5,2) { };
   
    \path (0)  edge [dashed] node { }(1)
          (1)  edge node { }(21) edge node { }(22)
            (21) edge node { }(22) edge node { }(31) edge node { }(33)           
          (22) edge node { }(31) 
          (22) edge node { }(32) %
          (31) edge node { }(32) 
          (32) edge [bend right=25] node { }(33)
         (31) edge node { }(33)
          (4)  edge node { }(32) edge node { }(33)  edge [dashed] node { }(5)
               
          (5)   edge [dashed] node { }(x1)        
          (x1)  edge node { }(x21) edge node { }(x22)
          (x21) edge node { }(x22) edge node { }(x32) edge node { }(x33)
         (x22) edge node { }(x32) 
         (x32) edge node { }(x33)
          (x4)  edge node { }(x32) edge node { }(x33)
                edge [dashed] node { }(x5)    ;
    \end{tikzpicture}
    \caption{The case when $x_3<\min(y_3,z_3)$, and $x_3 \geq y_2$, before  and after the operations (left and right).}
    \label{fig:box2}
\end{figure} 
        
        Note that $|L_{i-1}\cup L_i|$ did not change. When we  verify the minimum degree condition, we use the notation of Case~\ref{casebox1}.  We have:
\begin{eqnarray*}
    d'(X_{i-2}) & = & d(X_{i-2}), \\
    d'(Y_{i-1}) &=& x_1+y_2+y_3+z_2+z_3 > x_1+z_2+x_3+z_3 = d(Y_i), \\
    d'(Z_{i-1}) &=& x_1+x_3+y_2+y_3 = d(Z_{i-1}) , \\
    d'(X_i) &=& x_3+x_4+z_2+z_3 = d(Y_i) , \\
    d'(Z_i) &=& x_3+x_4+y_2+y_3 = d(Z_i) , \\
    d'(Y_{i+1}) &=& d(X_{i+1}) +y_2  .
\end{eqnarray*}
At this point, we are left with checking the satement for $x_3<\min (y_3,z_3,y_2, z_2)$. We split this into two cases.
    \item $x_3<\min (y_3,z_3,y_2, z_2)$ and $z_2\geq y_3$.\\ 
The operations are shown in Figure \ref{fig:box31}.
When we check the degrees, we use the notation introduced in Case~\ref{casebox1}:
\begin{eqnarray*}
d'(X_{i-2})&=& d(X_{i-2})+x_3 ,\\
d'(Y_{i-1})&=&x_1+z_2 +y_3+z_3> x_1+z_2+x_3 +z_3      =d(Y_{i-1}),\\
d'(Z_{i-1})&=&x_1+y_2 +x_3+ y_3+z_3>   x_1+y_2+y_3+x_3   =d(Z_{i-1}),\\
d'(X_i)&=&z_2+y_2+x_3 +x_4\geq y_2+y_3+x_3+x_4=d(Z_i) ,\\
d'(Y_{i+1})&=&d(X_{i+1}) .
\end{eqnarray*}
  
\begin{figure}[ht]
    \centering
    \begin{tikzpicture}[-,auto,node distance=2.25cm, thick, scale=0.7]
    \tikzstyle{every state}=[fill=none,draw=black,text=black,scale=0.6]
    \node[state] (0)  [draw=white] at (-1.5,0){ };
    \node[state] (1)  [label={\small $x_1$}] at (0,0) {$X$};
    \node[state] (21) [label={\small $y_2$}] at (2,2)  {$Y$};
    \node[state] (22) [label=below:{\small $z_2$}] at (2,-2) {$Z$};
    \node[state] (32) [label={\small $y_3$}] at ( 4,2) {$Y$};
    \node[state] (31) [label={[xshift=2.5mm]\small $x_3$}] at (4,0) {$X$};
    \node[state] (33)  [label=below:{\small $z_3$}]  at (4,-2) {$Z$};
    \node[state] (4)  [label={\small $x_4$}] at (6,0) {$X$};
    \node[state] (5)  [draw=white] at (7.5,0) { };
   
    \node[state] (x1)  [label={\small $x_1$}] at (9,0) {$X$};
    \node[state] (x21) [label={[label distance=1mm]\small $y_2+x_3$}] at (11,2) {$Y$};
    \node[state] (x22) [label=below:{\small $z_2$}] at (11,-2) {$Z$};
    \node[state] (x31)  [label={[label distance=1mm]\small $y_3+z_3$}] at (13,0) {$X$};
    \node[state] (x4)  [label={\small $x_4$}]  at (15,2) {$Y$};
    \node[state] (x5)  [draw=white] at (16.5,2) { };
   
    \path (0)  edge [dashed] node { }(1)
          (1)  edge node { }(21) edge node { }(22)
            (21) edge node { }(22) edge node { }(31) edge node { }(33)           
          (22) edge node { }(31) 
          (22) edge node { }(32) %
          (31) edge node { }(32) 
          (32) edge [bend right=25] node { }(33)
         (31) edge node { }(33)
          (4)  edge node { }(32) edge node { }(33)  edge [dashed] node { }(5)
               
          (5)   edge [dashed] node { }(x1)        
          (x1)  edge node { }(x21) edge node { }(x22)
          (x21) edge node { }(x22) edge node { }(x31)
          (x22) edge node { }(x31) 
          (x4)  edge node { }(x31) edge [dashed] node { }(x5) ;
    \end{tikzpicture}
    \caption{The case  $x_3\le\max(y_2,y_3,z_2,z_3)$ and  $z_2\geq y_3$, before and after the operations.}
    \label{fig:box31}
\end{figure}
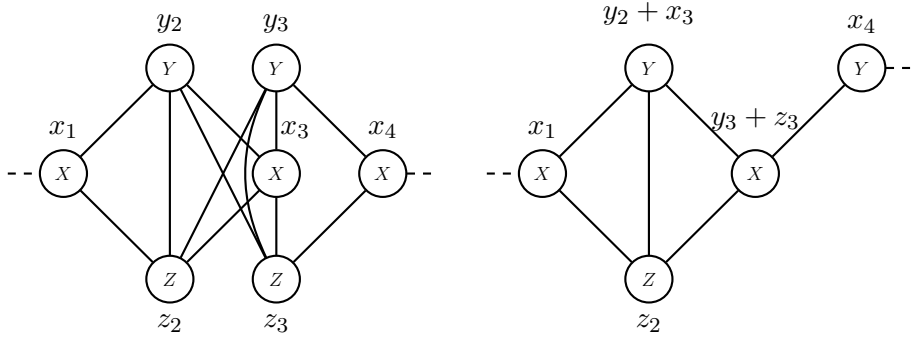

        \item $x_3<\min (y_3,z_3,y_2, z_2)$ and $z_2< y_3$.\\
        First note that the clump colored $X$ in $L_i$ can simply be moved into $L_{i-1}$ keeping the minimum degree condition. After this move on 
        the left side of Figure \ref{fig:box32}, we see the mirror image of the left side of Figure \ref{fig:box31}, just the numbers are different.
        The operations for this case, which are  the ``mirror image", of the operations in the previous case, is shown in Figure \ref{fig:box32}. 
        Because of the symmetry, we do not delve into the details. 
           
   \begin{figure}[ht]
    \centering
    \begin{tikzpicture}[-,auto,node distance=2.25cm, thick, scale=0.7]
    \tikzstyle{every state}=[fill=none,draw=black,text=black,scale=0.6]
    \node[state] (0)  [draw=white] at (-1.5,0){ };
    \node[state] (1)  [label={\small $x_1$}] at (0,0) {$X$};
    \node[state] (21) [label={\small $y_2$}] at (2,2)  {$Y$};
    \node[state] (22) [label=below:{\small $z_2$}] at (2,-2) {$Z$};
    \node[state] (32) [label={\small $y_3$}] at ( 4,2) {$Y$};
    \node[state] (31) [label={[xshift=2.5mm]\small $x_3$}] at (4,0) {$X$};
    \node[state] (33)  [label=below:{\small $z_3$}]  at (4,-2) {$Z$};
    \node[state] (4)  [label={\small $x_4$}] at (6,0) {$X$};
    \node[state] (5)  [draw=white] at (7.5,0) { };
   
    \node[state] (x1)   [label=below:{\small $x_1$}] at (9,0) {$X$};
    \node[state] (x21) [label={\small $y_2+z_2$}] at (11,2) {$Y$};
    \node[state] (x32)  [label={[label distance=2mm]\small $y_3+x_3$}] at (13,0) {$X$};
    \node[state] (x33)  [label=below:{\small $z_3$}] at (13,-2) {$Z$};
    \node[state] (x4)   [label={\small $x_4$}] at (15,2) {$Y$};
    \node[state] (x5)  [draw=white] at (16.5,2) { };
   
    \path (0)  edge [dashed] node { }(1)
          (1)  edge node { }(21) edge node { }(22)
            (21) edge node { }(22) edge node { }(31) edge node { }(33)           
          (22) edge node { }(31) 
          (22) edge node { }(32) %
          (31) edge node { }(32) 
          (32) edge [bend right=25] node { }(33)
         (31) edge node { }(33)
          (4)  edge node { }(32) edge node { }(33)  edge [dashed] node { }(5)
               
          (5)   edge [dashed] node { }(x1)        
          (x1)  edge node { }(x21) 
          (x21) edge node { }(x32) edge node { }(x33)
          (x32) edge node { }(x33) 
          (x4)  edge node { }(x32) edge node { }(x33)
                edge [dashed] node { }(x5) ;
    \end{tikzpicture}
    \caption{The case $x_3\le\max(y_2,y_3,z_2,z_3)$ and  $z_2< y_3$, before and after the operations.}
    \label{fig:box32}
\end{figure}
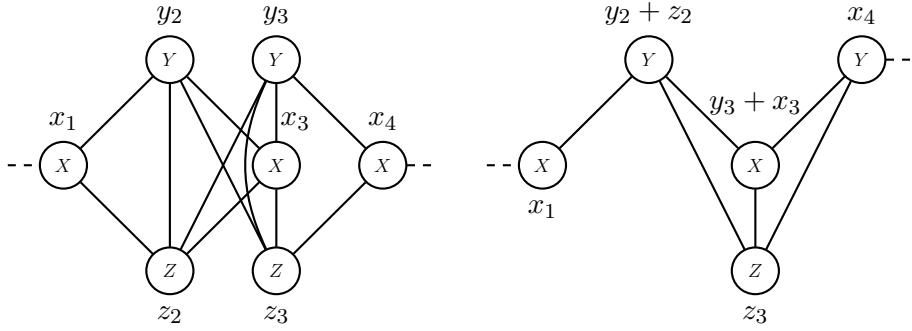        
\end{enumerate}
This concludes the proof of \ref{part:k1}

Finally, to prove part \ref{part:singleton},
take  the least  $i$, such that  $|L_i|>c(i)$, but $c(i)+\max\bigl(c(i-1),c(i+1)\bigl)< k$. 
As $|L_0|=c(0)=1$, this means $i>0$. 
First we show that
we may assume that $L_{i-1}\cup L_{i}\cup L_{i+1}$ misses some color $X$. Indeed, if all colors appear in
$L_{i-1}\cup L_{i}\cup L_{i+1}$, let $X$ be a color not used in $L_{i-1}\cup L_{i}$ and $Y$ be a color not used in $L_{i}\cup L_{i+1}$.
Create a new coloring of $G$ by switching the colors $X,Y$ in $L_j$ for all $ j\ge i+1$. After the switch, $X$ is missing from $L_{i-1}\cup L_{i}\cup L_{i+1}$.

Since  $|L_i|>c(i)$, there are two vertices $x,y$ in $L_i$ that are colored the same. Recolor $x$ with color $X$. This is a valid coloring, in which
$L_i$ contains one more color then before. Repeating this procedure produces a coloring, in which $|L_i|=c_i$ or $c(i)+\max\bigl(c(i-1),c(i+1)\bigl)= k$. Repeating this procedure 
recursively for the next  least $i$, we can eliminate one after the other the $i$'s that fail  \ref{part:singleton}, not creating any instances where the first three statements
would fail.
\end{proof}

 \begin{definition}\label{def:canon} We call a $k$-colored weighted clump graph 
$H$ \emph{canonical}, if  there is a graph $G$, whose clump graph is $H$, and  $H$ satisfies the four statements in Theorem~\ref{th:canonical}, i.e.,
$H$ has $D+1$ layers $L_0,L_1,\ldots,L_D$, where $D=\diam(H)$, and for each $1\le i<D$ we have 
\begin{enumerate}[label={\upshape (\roman*)}]
\item\label{part:H1k} If $|L_i|=1$, then $L_{i+1}\le k-1$.
\item\label{part:Hmanycolor} The number of colors used to color the set $L_i\cup L_{i+1}$ is $\min\bigl(k,c(i)+c(i+1)\bigl)$. In particular, when
$c(i)+c(i+1)\le k$, then $L_i$ and $L_{i+1}$ do not share any color.
\item\label{part:Hk1} If $|L_i|=k$,  then $i\ge 2$ and $|L_{i+1}|\ge 2$.
\item\label{part:Hsingleton} If $L_i$ has a weight that is bigger than $1$, then $i>0$ and $|L_i|+\max\bigl(|L_{i-1}|,|L_{i+1}|\bigl)\geq k$.
\end{enumerate}
\end{definition}
Note that \ref{part:Hmanycolor} implies that the edges missing between $L_i$ and $L_{i+1}$ form a matching of size $\max\bigl(0,|L_i|+|L_{i+1}|-k\bigl)$. In particular, when $|L_i\cup L_{i+1}|\leq k$ then all edges between $L_i$ and $L_{i+1}$ are present.

\begin{corollary} \label{3canon}
In the canonical clump graph of a 3-colored connected graph, the following color sets are possible in two consecutive layers:
\begin{equation} \label{possiblepatterns}
X\Bigl| Y, \, \, X\Bigl| YZ,  \, \, YZ\Bigl| X,  \, \, XY\Bigl| XZ,  \, \, XY\Bigl| XYZ,  \, \, XYZ\Bigl| XY,  \, \, XYZ\Bigl| XYZ.
\end{equation}
\end{corollary}

\section{Duality}
\label{sec:dual}


In this Section, $k$ is fixed.
Look differently at our diameter problem: assume that the diameter $D$,  and the lower bound $\delta$ for the degrees of the graph
are  fixed (in addition to $k$), how small $n$ can be, such that connected $k$-colorable graphs of order
$n$, minimum degree at least $\delta$, and diameter $D$ exist? Let $\mathcal{H}$   denote the family of canonical clump graphs of 
diameter $D$ that arises from connected $k$-colorable graphs  with diameter $D$ and minimum degree at least $\delta$, of unspecified order. Fix an $H\in \mathcal{H}$, 
  and consider the following packing problem for $H$:
 assign non-negative real dual weights $u(y)\geq 0$ to $y\in V(H)$, and  
$$ \text{Maximize } \delta \cdot \sum_{y\in V(H)} u(y),$$
\noindent subject to condition
\begin{equation}
    \label{eq:dualcond}
 \forall x\in V(H) \qquad  \sum\limits_{y\in V(H):xy\in E(H)} u(y) \leq 1. 
\end{equation}

 \begin{theorem} \label{th:sulyozas}
 Assume that  there exist constants $\tilde u>0, C\geq 0$, such that for all $D$ and $\delta$,  and all $H\in \mathcal{H}$, in the linear program \emph{(\ref{eq:dualcond})}  the optimum is at least 
 \begin{equation}\label{feltettuk}
 \tilde u \delta (D+1) -C\delta.  \end{equation}
Then, for any $H\in \mathcal{H}$,
  we have 
 \begin{equation} \label{tildes}
 D\leq \frac{1}{\tilde u}\cdot \frac{n}{\delta} + C.
 \end{equation} 
 \end{theorem}

\begin{proof} Fix $H\in \mathcal{H}$. Then $H$ is the clump graph of saturated graph $G$. $G$ can be reconstructed by 
assigning $w(x)\geq 1$ integer weights
for all vertices of $H$, such that we assign 1 to the vertex in $L_0$. Now $n=|V(G)|=\sum\limits_{x\in V(H)} w(x)$. Consider the optimization problem 
$$ \text{Minimize }\sum\limits_{x\in V(H)} w(x),$$ 
\noindent subject to condition 
\begin{equation}  \label{eq:primalcond}
\forall y\in V(H)\qquad \sum\limits_{x\in V(H):xy\in E(H)} w(x) \geq \delta.
\end{equation}
We face the trivial inequality of the duality of linear programming \cite{dantzig}: namely, for any $u $ and $w$ feasible solutions, by (\ref{eq:primalcond}) and 
(\ref{eq:dualcond}), we have:
\begin{eqnarray*}
\delta \sum_{y\in V(H)} u(y) & \leq &  \sum_{y\in V(H)} u(y)\sum_{x\in V(H): xy\in E(H)} w(x) \\
& = & \sum_{x\in V(H)} w(x) \sum_{y\in V(H): xy\in E(H)} u(y) \leq  \sum_{x\in V(H)} w(x).
\end{eqnarray*}
As the objective function reaches $\delta\sum_{y\in V(H)} u(y) \geq \tilde u \delta (D+1) -C\delta$, the theorem follows.
\end{proof}
\vskip 12pt

\noindent \textit{Proof of Theorem~\ref{th:upperbound}}.  
Assume $k\geq 3$. Consider a $k$-colorable canonical clump graph $H$ with
layers $L_0,\ldots,L_D$. . 
We are going to  find a good packing $u$ on the vertices of $H$ as required to use        
Theorem~\ref{th:sulyozas}. The dual weighting $u$ will take at most $2k-2$ different values, and every layer will get the same total dual weight.

Let $i$ be an integer, $0\leq i\leq D$. Set $L_{-1}=L_{D+1}=\emptyset$.

If $|L_i|\leq k-1$, then assign the dual weight $\frac{k-1}{(3k-4)|L_i|}$ to every $v\in L_i$. This makes the total dual weight of $L_i$ exactly $\frac{k-1}{3k-4}$,
and the dual weight of every $v\in L_i$ at least $\frac{1}{3k-4}$. 
 
 If $|L_i|=k$, 
 let $X_i$ be the (possibly empty) set of vertices in $L_i$  connected to every vertex in $L_{i-1}\cup L_{i+1}$, and set  $Y_i=L_i\setminus X_i$.
As $|X_i|\le k-|L_{i-1}|$ and $H$ is canonical, by Definition~\ref{def:canon}\ref{part:H1k} we have $0\le |X_i| \le k-2$, and $Y_i\ne\emptyset$.

Set the dual weight of every $v\in X_i$ to $\frac{1}{3k-4}$, and  the dual weight
of every $v\in Y_i$ to $\frac{1}{3k-4}-\frac{1}{(3k-4)(k-|X_i|)}$. As $|Y_i|=k-|X_i|$, the total dual weight of $L_i$ is
$$\frac{|X_i|}{3k-4}+(k-|X_i|)\left(\frac{1}{3k-4}-\frac{1}{(3k-4)(k-|X_i|)}\right)=\frac{k-1}{3k-4}.$$
Moreover, as $k-|X_i|\ge 2$, the dual weight of $v\in Y_i$ is at least $\frac{1}{2(3k-4)}$.
 
Now take vertex $x$ of $H$. Then $x\in L_j$ for some $0\le j\le D$. We are going to check that the neighbors of $x$ have a total dual weight of at most 1. 

If  $|L_j|\leq k-1$, or ($|L_j|=k$ and $x\in X_j$), then the weight of $x$ is at least $\frac{1}{3k-4}$. Since the open neighborhood  of
$x$ is a subset of $(L_{j-1}\cup L_j\cup L_{j+1})\setminus\{x\}$,
the sum of the weight of its neighbors is at most
$\frac{3(k-1)}{3k-4}-\frac{1}{3k-4}=1$.

If $|L_j|=k$ and $x\in Y_j$, then there is a $y\in L_{j-1}\cup L_{j+1}$ such that the open neighborhood  of $x$ is contained by
 $(L_{j-1} \cup L_j\cup L_{j+1})\setminus\{x,y\}$. As the sum of the weights of $x$ and $y$ is at least $\frac{1}{3k-4}$,
 the total weight of the neighbors of $x$ at at most $\frac{3(k-1)}{3k-4}-\frac{1}{3k-4}=1$.      

The total dual weight of the vertices in $H$ is $\frac{k-1}{3k-4}(D+1)$. 
 Now  Theorem~\ref{th:upperbound}
follows from Theorem~\ref{th:sulyozas}.
\hfill$\qed$

\vskip 12pt

Using $k=3$, we get  a weaker version of Theorem \ref{th:CDS}:

\begin{corollary}
\label{th:5o2}
If $G$ is connected $3$-colorable graph of order $n$ and minimum degree
$\delta\ge 1$,
$$ \diam(G) \leq \frac{5n}{2\delta}+O(1). $$
\end{corollary}

\section{Inclusion-Exclusion (Sieve)}
\label{sec:inex}

Let us be given a 3-colorable saturated connected graph  $G$  of order $n$ and minimum degree at least  $\delta$, which maximizes the 
diameter $D$ among such graphs. By Theorem~\ref{th:canonical}, we may assume without loss of generality that the clump graph of $G$ is canonical.
Furthermore, Corollary~\ref{3canon} tells what kind of color sets can be in consecutive layers.
We often use these facts without explicit reference in the future.
Let $\ell_i=|L_i|$ denote the cardinality of the $i^{\text{th}}$ layer of $G$. 
As we are about to prove Theorem~\ref{posmain}, we can assume without loss of generality 
that $\ell_i\leq 3\delta$. Indeed, if $G$ does not satisfy this inequality, eliminate vertices from clumps with excess above $\delta$, to obtain the graph $G'$ on $n'$ vertices.
$G'$ still satisfy the conditions of Theorem~\ref{posmain}, and therefore its conclusion with $n'$ replacing $n$. Hence $G$ also satisfies the conclusion of
Theorem~\ref{posmain}.
We are going to build lower bounds for the sum of a couple of consecutive $\ell_i$'s, from
which we derive lower bounds for $n$. The key tool is the inclusion-exclusion formula for the size of the union of the open neighborhoods of some 
vertices.  Note that a vertex in $L_i$ can have neighbors only in $L_{i-1},L_i,L_{i+1}$. We denote the open neighborhood of vertex $z$ by $N(z)$.
 In Subsection~\ref{onelayer} we do this approach when the vertices are taken from different clumps from the same $L_i$,
 in Subsection~\ref{twolayers} we do this for vertices  taken from two consective layers.
 Recall that $c(i)$ denotes the number of clumps in $L_i$. Let $\mathcal{S}=\{i: c(i)=1\}$ be the set  of \emph{singles}.
We use the notation $x_i,y_i,z_i$ to represent  vertices in the clumps with color $X_i, Y_i, Z_i$, respectively. Here $X_i, Y_i, Z_i$ can be any of the colors
$A,B,C$, but they must be different colors. For the ease of computation we introduce $L_{-1}=L_{D+1}=\emptyset$, so $\ell_{-1}=\ell_{D+1}=0$.

\subsection{Sieve for neighborhoods of vertices from one layer} \hfill
\label{onelayer}

Here we assume $0\le i\le D$.

\noindent \underline{Case 1.} $c(i)=1$. We obviously have $\ell_{i-1}+\ell_{i+1}\geq \delta$, which we prefer to write as
\begin{equation} \label{oneperone} 
2\ell_{i-1}+2\ell_{i}+2\ell_{i+1}\geq 2\delta +2\ell_i.
\end{equation}
\underline{Case 2.} We have $2\ell_{i-1}+\ell_{i}+2\ell_{i+1}\geq 2\delta$ from the fact that vertices from either color in the
$i\textsuperscript{th}$ layer have at least $\delta $ neighbors. We prefer to write this as
\begin{equation} \label{onepertwo}
2\ell_{i-1}+2\ell_{i}+2\ell_{i+1}\geq 2\delta +\ell_i.
\end{equation}

\subsection{Sieve by two consecutive layers} \hfill
\label{twolayers}

Now we assume $0\le i<D$, so $i+1\le D$.

\noindent \underline{Case 1.} $i\in \mathcal{S}, i+1\in \mathcal{S}$. We have
\begin{equation}\label{twoperone}
\ell_{i-1}+\ell_i+\ell_{i+1}+\ell_{i+2}\geq 2\delta.
\end{equation}
\noindent \underline{Case 2.} $i\in \mathcal{S},i+1\notin \mathcal{S}$. Assume $X_i=L_i$, which implies $L_{i+1}=Y_{i+1}\cup Z_{i+1}$. 
 Apply (\ref{oneperone}) to $L_i$ to obtain $2\ell_{i-1}+2\ell_{i+1}\geq 2\delta$, apply 
(\ref{onepertwo}) to $L_{i+1}$ to obtain $2\ell_{i}+\ell_{i+1}+2\ell_{i+2}\geq 2\delta$, and average into
\begin{equation} \label{twopertwo'}
\ell_{i-1} +\ell_i+\frac{3}{2}\ell_{i+1}+\ell_{i+2}\geq 2\delta.
\end{equation} 
\noindent \underline{Case 3.}  $i\notin \mathcal{S},i+1\in \mathcal{S}$. Like in Case 2, we obtain
\begin{equation} \label{twoperthree'}
\ell_{i-1} +\frac{3}{2}\ell_i+\ell_{i+1}+\ell_{i+2}\geq 2\delta.
\end{equation} 
\noindent \underline{Case 4.} $i\notin \mathcal{S},i+1\notin \mathcal{S}$. In this case $L_i$ and $L_{i+1}$ must share a color, and their union must use all $3$ colors. We can assume without loss of generality that none of $X_i$, $Y_i $, 
$X_{i+1}$, $Z_{i+1} $ is empty. Take $x_i\in X_i, y_i\in Y_i, x_{i+1}\in X_{i+1}, z_{i+1}\in Z_{i+1}$.
 Considering the neighborhood of $x_i$, we have 
 \begin{equation} \label{elso}
 \delta\leq \ell_{i-1}+|Y_i|+|Z_i|+|Y_{i+1}|+|Z_{i+1}|,
 \end{equation}
considering the neighborhood of $x_{i+1}$, we have 
\begin{equation} \label{masodik}
 \delta\leq \ell_{i+2}+|Y_i|+|Z_i|+|Y_{i+1}|+|Z_{i+1}|,
 \end{equation}
considering the neighborhood of $y_{i}$, we have 
\begin{equation} \label{harmadik}
 \delta\leq \ell_{i-1}+|X_i|+|Z_i|+|X_{i+1}|+|Z_{i+1}|,
 \end{equation} 
considering the neighborhood of $z_{i+1}$, we have 
\begin{equation} \label{negyedik}
 \delta\leq \ell_{i+2}+|X_i|+|Y_i|+|X_{i+1}|+|Y_{i+1}|.
 \end{equation}
 Weighting (\ref{elso}) and (\ref{masodik})  with $1/3$, (\ref{harmadik}) and (\ref{negyedik}) with $2/3$, and summing them up, 
 we obtain
\begin{equation} \label{twoperfour}
\ell_{i-1} +\frac{4}{3}[\ell_i+\ell_{i+1}]+\ell_{i+2}\geq 2\delta.
\end{equation} 
Adding up  (\ref{twoperone}), (\ref{twopertwo'},) (\ref{twoperthree'}), (\ref{twoperfour}) for $i=1,\ldots, D-1$,  we obtain
\begin{equation} \label{pairsummary'}
4n+\sum_{\substack{(i,j): i\notin S\\
   j\notin S, |i-j|=1}} \frac{1}{3}\ell_i + \sum_{i: i+1\in S, i\notin S} \frac{1}{2}\ell_i 
+ \sum_{i: i-1\in S, i\notin S} \frac{1}{2}\ell_i    
       \geq 2D\delta  +O(\delta).
\end{equation}
The $O(\delta)$ error term arises from the fact that certain $\ell_i$ terms,  at the front and at the end, do not arise four times, as many times they are counted in $4n$.

\vspace{2mm}
\subsection{Sieve  for neighborhoods of vertices from three consecutive layers} \hfill
\label{threelayers}

\vspace{3pt}
\noindent We are going to give lower bounds to
\begin{equation} \label{otos}
2( \ell_{i-2}+ \ell_{i-1}+ \ell_{i}+ \ell_{i+1}+ \ell_{i+2})=2|L_{i-2}\cup L_{i-1}\cup L_i\cup L_{i+1}\cup L_{i+2}|
\end{equation}
using inclusion-exclusion, based on a case analysis of the color content of $L_{i-1},L_i,L_{i+1}$.

\noindent \underline{Case 1.} $i-1\notin \mathcal{S},i\notin \mathcal{S}, i+1\notin \mathcal{S}$. This boils down to two subcases:\\
\underline{Subcase 1.1.} $L_{i-1}$ and $L_{i+1}$ share at least two colors. We may assume in this case that none of $X_{i-1}$, $Y_{i-1}$, $X_{i}$, $Z_{i}$, $X_{i+1}$, $Y_{i+1}$ is empty.\\
Take $y_{i-1}\in Y_{i-1}, z_i\in Z_i, x_{i+1}\in X_{i+1}$. Using inclusion-exclusion we have 
$$|N(y_{i-1})\cup N(z_i)\cup N(x_{i+1})|\geq 3\delta - (|X_{i-1}|+\ell_i+|Y_{i+1}|).$$
Similarly, take $x_{i-1}\in X_{i-1}, z_i\in Z_i, y_{i+1}\in Y_{i+1}$ and use inclusion-exclusion to get 
$$|N(x_{i-1})\cup N(z_i)\cup N(y_{i+1})|\geq 3\delta- (|Y_{i-1}|+\ell_i+|X_{i+1}|).$$
Combining the two inequalities above we obtain
\begin{equation} \label{local}
2( \ell_{i-2}+ \ell_{i-1}+ \ell_{i}+ \ell_{i+1}+ \ell_{i+2})\geq 6\delta-2\ell_i-\ell_{i-1}-\ell_{i+1}.
\end{equation}
\underline{Subcase 1.2.} $L_{i-1}$ and $L_{i+1}$ share only one color. We may assume
that $L_{i-1}=X_{i-1}\cup Y_{i-1}$;  $L_{i+1}=Y_{i+1}\cup Z_{i+1}$ where $X_{i-1},Y_{i-1},Y_{i+1}, Z_{i+1}\neq\emptyset$, and $X_{i},Z_{i}\neq\emptyset$.\\
Apply inclusion-exclusion for the neighborhoods of $x_{i-1}\in X_{i-1}, z_i\in Z_{i}, y_{i+1}\in Y_{i+1}$ we get
$$|N(x_{i-1})\cup N(z_i)\cup N(y_{i+1})|\geq 3\delta - (|Y_{i-1}|+\ell_i+|Z_{i+1}|),$$
and doing it again for  $y_{i-1}\in Y_{i-1}, x_i\in X_i, z_{i+1}\in Z_{i+1}$ we get
$$|N(y_{i-1})\cup N(x_i)\cup N(z_{i+1})|\geq 3\delta- (|Z_{i-1}|+\ell_i+|Y_{i+1}|).$$
We obtain (\ref{local}), like in the previous subcase.\\
\underline{Case 2.} $i-1\notin \mathcal{S},i\in \mathcal{S}, i+1\notin \mathcal{S}$. 
We may assume $L_i=Z_i$, and for $j\in\{i-1,i+1\}$ $L_{j}=X_j\cup Y_j$, where none of $X_{i+1}$, $Y_{i+1}$, $X_{i-1}$, $Y_{i-1}$, $Z_i$ is empty. 
This can be handled like Subcase 1.1 to obtain (\ref{local}).\\
\underline{Case 3.} $i-1\in \mathcal{S},i\in \mathcal{S}, i+1\in \mathcal{S}$. We can assume $L_{i-1}=X_{i-1}, L_i=Y_i, L_{i+1}=Z_{i+1}$ (in case  $L_{i+1}=X_{i+1}$, switch colors $X$ and $Z$
in layers $L_j$ for  $j\geq i+1$). Select $x_{i-1}\in X_{i-1}, y_i\in Y_i, z_{i+1}\in Z_{i+1}$, and apply inclusion-exclusion for 
$|N(x_{i-1})\cup N(y_i) \cup N(z_{i+1}) |$ to obtain 
\begin{equation*} 
2( \ell_{i-2}+ \ell_{i-1}+ \ell_{i}+ \ell_{i+1}+ \ell_{i+2})\geq 6\delta-2\ell_i.
\end{equation*}
\underline{Case 4.} $i-1\in \mathcal{S},i\in \mathcal{S}, i+1\notin \mathcal{S}$. 
As the clump graph is canonical, $c(i+1)=2$. Hence we can assume  $L_{i-1}=X_{i-1}, L_i=Y_i, L_{i+1}=
X_{i+1}\cup Z_{i+1}$. Applying inclusion-exclusion for the neighborhoods of representative elements,
we obtain 
$$|N(x_{i-1})\cup N(y_i)\cup N(x_{i+1})|\geq 3\delta -\ell_i-|Z_{i+1}|$$
and 
$$|N(x_{i-1})\cup N(y_i)\cup N(z_{i+1})|\geq 3\delta -\ell_i-|X_{i+1}|.$$
Combining the last two displayed formulae, we obtain  
\begin{equation*} 
2( \ell_{i-2}+ \ell_{i-1}+ \ell_{i}+ \ell_{i+1}+ \ell_{i+2})\geq 6\delta-2\ell_i-\ell_{i+1},
\end{equation*}
which is even stronger than (\ref{local}).\\
\underline{Case 5.} $i-1\notin \mathcal{S},i\in \mathcal{S}, i+1\in \mathcal{S}$. This is a mirror image of Case 4, so we have
\begin{equation*} 
2( \ell_{i-2}+ \ell_{i-1}+ \ell_{i}+ \ell_{i+1}+ \ell_{i+2})\geq 6\delta-\ell_{i-1}- 2\ell_i.
\end{equation*}
\underline{Case 6.} $i-1\in \mathcal{S},i\notin \mathcal{S}, i+1\in \mathcal{S}$. 
We may assume $X_{i-1}=L_{i-1}$, $Y_i\cup Z_i=L_i$, $X_{i+1}=L_{i+1}$, where $X_{i-1}$, $Y_i$, $Z_i$, $X_{i+1}$ are nonempty. Select $x_{i-1}\in X_{i-1}$, $y_i\in Y_i$, $z_i\in Z_i$,
$x_{i+1}\in X_{i+1}$. Clearly
$$
|L_{i-2}\cup L_{i-1}\cup L_i\cup L_{i+1}\cup L_{i+2}|\geq
$$
$$ \geq |N(x_{i-1})\cup N(x_{i+1})| +|N(y_i)\cup N(z_i)|- |(N(x_{i-1})\cup N(x_{i+1}))\cap (N(y_i)\cup N(z_i))|\geq
$$
$$
\geq (2\delta-\ell_i) +(2\delta-\ell_{i-1}-\ell_{i+1})-\ell_i =4\delta -2\ell_i-\ell_{i-1}-\ell_{i+1}.
$$
We conclude
\begin{equation} \label{negyszog}
2( \ell_{i-2}+ \ell_{i-1}+ \ell_{i}+ \ell_{i+1}+ \ell_{i+2}) 
\geq 8\delta -4\ell_i-2\ell_{i-1}-2\ell_{i+1}.
\end{equation} 
\noindent \underline{Case 7.} $i-1\in \mathcal{S},i\notin \mathcal{S}, i+1\notin \mathcal{S}$.
We may assume $X_{i-1}=L_{i-1}$, $Y_i\cup Z_i=L_i$, where $X_{i-1}$, $Y_i$, $Z_i$ and $X_{i+1}$ are nonempty. Select $x_{i-1}\in X_{i-1}$, $y_i\in Y_i$, $z_i\in Z_i$,
$x_{i+1}\in X_{i+1}$. Clearly
\begin{eqnarray*}
    \abs*{L_{i-2}\cup L_{i-1}\cup L_{i}\cup L_{i+1}\cup L_{i+2}} 
    &\geq& \abs*{N(x_{i-1})\cup N(x_{i+1})\cup N(y_i)\cup N(z_i)}\\
    &\geq& (2\delta-\ell_i) +(2\delta-\ell_{i-1}-\ell_{i+1})-\ell_i\\ 
    &=&4\delta -2\ell_i-\ell_{i-1}-\ell_{i+1}.
\end{eqnarray*}
We conclude (\ref{negyszog}) again.\\
\underline{Case 8.} $i-1\notin \mathcal{S},i\notin \mathcal{S}, i+1\in \mathcal{S}$. As this is the mirror image of Case 7', we arrive at the same conclusion  (\ref{negyszog}),
as the conclusion is symmetric.

\vspace{2mm}
For $1\leq i\leq D-1$, we call a triplet of consecutive layers $(i-1,i,i+1)$ \emph{singular}, if $i\notin S$ and ($ i+1\in S$ or $i-1\in S$). 
Let $s$  denote the number of singular  triplets. 
Summing up the lower bounds to (\ref{otos}) obtained in the 8 cases, we have

\begin{eqnarray*}
    10n&\geq & \; 6\delta D-2n+O(\delta) -\sum_{\substack{i:i\notin\mathcal{S}\\ i-1,i+1\notin\mathcal{S}}}(\ell_{i-1}+\ell_{i+1})
    -\sum_{\substack{i:i\in\mathcal{S}\\ i-1,i+1\notin\mathcal{S}}}(\ell_{i-1}+\ell_{i+1})\\
    &&\;\;
  -\sum_{\substack{i\in\mathcal{S}\\ i-1\in\mathcal{S},i+1\notin\mathcal{S}}}\ell_{i+1}
  -\sum_{\substack{i\in\mathcal{S}\\ i-1\notin\mathcal{S},i+1\in\mathcal{S}}}\ell_{i-1}
    -\sum_{\substack{i:i\notin\mathcal{S}\\ i-1\in\mathcal{S}\vee i+1\in\mathcal{S}}}(-2\delta+2 \ell_i+\ell_{i-1}+\ell_{i+1})\\
&=&\; 6\delta D-2n+2s\delta +O(\delta) 
-\sum_{\substack{i:i\notin\mathcal{S}\\ i+1,i+2\notin\mathcal{S}}}\ell_i
-\sum_{\substack{i:i\notin\mathcal{S}\\ i-1,i-2\notin\mathcal{S}}}\ell_i
-\sum_{\substack{i:i\notin\mathcal{S}\\ i+1\in\mathcal{S},i+2\notin\mathcal{S}}}\ell_i
-\sum_{\substack{i:i\notin\mathcal{S}\\ i-1\in\mathcal{S}, i-2\notin\mathcal{S}}}\ell_i\\
   &&\;\;
  -\sum_{\substack{i\notin\mathcal{S}\\ i-1,i-2\in\mathcal{S}}}\ell_{i}
   -\sum_{\substack{i\notin\mathcal{S}\\ i+1,i+2\in\mathcal{S}}}\ell_{i}
 -2\sum_{\substack{i:i\notin\mathcal{S}\\ i-1\in\mathcal{S}\vee i+1\in\mathcal{S}}} \ell_i   
 -\sum_{\substack{i:i+1\notin\mathcal{S}\\ i\in\mathcal{S}\vee i+2\in\mathcal{S}}} \ell_i
 -\sum_{\substack{i:i-1\notin\mathcal{S}\\ i\in\mathcal{S}\vee i-2\in\mathcal{S}}} \ell_i
 \end{eqnarray*}
 Now we use that
 \begin{eqnarray*}
 \sum_{\substack{i:i+1\notin\mathcal{S}\\ i\in\mathcal{S}\vee i+2\in\mathcal{S}}} \ell_i
 +\sum_{\substack{i:i-1\notin\mathcal{S}\\ i\in\mathcal{S}\vee i-2\in\mathcal{S}}} \ell_i\le
 \sum_{\substack{i:i\notin\mathcal{S}\\ i+1\notin\mathcal{S},i+2\in\mathcal{S}}}\ell_i+
  \sum_{\substack{i:i\notin\mathcal{S}\\ i-1\notin\mathcal{S},i-21\in\mathcal{S}}}\ell_i+ 
 2 \sum_{i:i\in\mathcal{S}}\ell_i
   \end{eqnarray*}
 and
 \begin{eqnarray*}
 2\sum_{i\notin\mathcal{S}}\ell_i&=& \left(\sum_{\substack{i:i\notin\mathcal{S}\\ i+1,i+2\notin\mathcal{S}}}\ell_i
+ \sum_{\substack{i:i\notin\mathcal{S}\\ i+1\in\mathcal{S},i+2\notin\mathcal{S}}}\ell_i
+\sum_{\substack{i:i\notin\mathcal{S}\\ i+1\notin\mathcal{S},i+2\in\mathcal{S}}}\ell_i
+\sum_{\substack{i:i\notin\mathcal{S}\\ i+1,i+2\in\mathcal{S}}}\ell_i
 \right)\\
 &&\;\;
  +\left(\sum_{\substack{i:i\notin\mathcal{S}\\ i-1,i-2\notin\mathcal{S}}}\ell_i
+ \sum_{\substack{i:i\notin\mathcal{S}\\ i-1\in\mathcal{S},i-2\notin\mathcal{S}}}\ell_i
+\sum_{\substack{i:i\notin\mathcal{S}\\ i-1\notin\mathcal{S},i-2\in\mathcal{S}}}\ell_i
+\sum_{\substack{i:i\notin\mathcal{S}\\ i-1,i-2\in\mathcal{S}}}\ell_i
 \right)
 \end{eqnarray*}
 to obtain
 \begin{eqnarray*}
12n &\geq&6\delta D+2s\delta +O(\delta) 
  -2\sum_{i:i\notin\mathcal{S}}\ell_i -2\sum_{\substack{i:i\notin\mathcal{S}\\ i-1\in\mathcal{S}\vee i+1\in\mathcal{S}}} \ell_i  
    -2\sum_{i:i\in{S}} \ell_i\\
 &=&6\delta D+2s\delta-2n +O(\delta) -2\sum_{\substack{i:i\notin\mathcal{S}\\ i-1\in\mathcal{S}\vee i+1\in\mathcal{S}}} \ell_i .   
 \end{eqnarray*}
This gives
\begin{eqnarray}
\label{Evatriples2}
7n\geq 3\delta D+s\delta +O(\delta) -\sum_{\substack{i:i\notin\mathcal{S}\\ i-1\in\mathcal{S}\vee i+1\in\mathcal{S}}} \ell_i.
\end{eqnarray}

\section{Optimization}
\label{sec:opt}

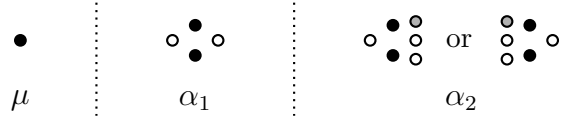
\begin{figure}[ht]
\centering
\begin{tikzpicture}[thick]
  \draw[black,fill=black] (1.9,0) circle (.7mm);
  \node at (1.9,-.8) {$\mu$};
     \draw[dotted] (2.9,.5)--(2.9,-1);
    \draw[black] (3.9,0) circle (.7mm);
    \draw[black,fill=black] (4.2,-.2) circle (.7mm);
    \draw[black,fill=black] (4.2,.2) circle (.7mm);
    \draw[black] (4.5,0) circle (.7mm);
    \node at (4.2,-.8) {$\alpha_1$};
    \draw[dotted] (5.5,.5)--(5.5,-1);
      \draw[black] (6.5,0) circle (.7mm);
    \draw[black,fill=black] (6.8,-.2) circle (.7mm);
    \draw[black,fill=black] (6.8,.2) circle (.7mm);
    \draw[black] (7.1,-.25) circle (.7mm);
    \draw[black] (7.1,0) circle (.7mm);
    \draw[black, fill=gray!60] (7.1,.25) circle (.7mm);
    \node[draw=white, text width=.4cm] at (7.7,0) {\small or};
    \draw[black] (8.3,-.25) circle (.7mm);
    \draw[black] (8.3,0) circle (.7mm);
    \draw[ black, fill=gray!60] (8.3,.25) circle (.7mm);
    \draw[black,fill=black] (8.6,-.2) circle (.7mm);
    \draw[black,fill=black] (8.6,.2) circle (.7mm);
    \draw[black] (8.9,0) circle (.7mm);
    \node at (7.7, -.8) {$\alpha_2$};   
\end{tikzpicture}
\caption{Visual representation for some variables denoted with Greek letters. Layers with black filled circles represent the layers whose vertices we count,   the empty circles show how many colors are present in the nearby layers. Gray filled circles represent a  third  color that may or may not be present
in the layer.
}
\label{fig:greek}
\end{figure}

The inequalities (\ref{pairsummary'})  and (\ref{Evatriples2}) are key  constraints for our linear program. 
The linear program is in global variables, which are mostly the fraction of vertices of $G$ in certain type of layers, which live in a neighborhood of certain type of layers.   The global  variables, denoted by Greek letters, will be:\\
\begin{eqnarray*}
\mu &=& \dfrac{1}{n} \sum\limits_{i:c(i)=1}\ell_i \\ 
\alpha_1 &=&\dfrac{1}{n} \sum\limits_{\substack{i:0<i<D, c(i)=2\\ i-1,i+1\in\mathcal{S}}}\ell_i\\
\alpha_2 &=&\dfrac{1}{n} \sum\limits_{\substack{i:0<i<D, c(i)=2, \\i-1\in\mathcal{S},i+1\notin\mathcal{S}}}\ell_i \quad + 
\quad \dfrac{1}{n} \sum\limits_{\substack{i:0<i<D, c(i)=2, \\i+1\in\mathcal{S}, i-1\notin\mathcal{S}}}\ell_i\\
\phi&=&\frac{D\delta}{n}\\
\psi&=&\frac{\delta s}{n}
\end{eqnarray*}
Figure~\ref{fig:greek} illustrate the variables whose definition involves sums.
Clearly, all variables are non-negative. 
We use Corollary~\ref{3canon} on what kind of layers can be consecutive.
From the definitions, it easily follows that
\begin{equation}\label{obvious}
\mu+\alpha_1+\alpha_2\leq 1\,\,\,\text{\ \  and\ \ }\,\,\,\psi\leq \frac{2}{3}.\\    
\end{equation}
We have
$$\sum_{\substack{(i,j): i\notin S\\
   j\notin S, |i-j|=1}}\ell_i =2n(1-\alpha_1-\alpha_2-\mu)+n\alpha_2+O(\delta)
   =n\left(2-2\mu -2\alpha_1-\alpha_2+O\left(\frac{\delta}{n}\right)\right), 
   $$
 since (except possibly for $i=D$) $\ell_i$'s accounted for in the definition of $\mu$ and $\alpha_i$ do not contribute to the sum on the left side,
 $\ell_i$-s accounted for in $\alpha_2$ appear once, and all other $\ell_i$'s appear twice.
In addition,
$$\sum_{i: i+1\in S, i\notin S} \ell_i 
+ \sum_{i: i-1\in S, i\notin S}\ell_i=n\left(2\alpha_1+\alpha_2+O\left(\frac{\delta}{n}\right)\right).$$
Using these observations,
simple algebra derives
from (\ref{pairsummary'})
\begin{equation}
    12\phi + 4\mu -2\alpha_1 -\alpha_2
    \leq 28 +O\left(\frac{\delta}{n}\right).
    \label{huszonnyolcas}
\end{equation}
From \ref{Evatriples2}, using
$$
\sum_{\substack{i:i\notin\mathcal{S}\\ i-1\in\mathcal{S}\vee i+1\in\mathcal{S}}} \ell_i.
=n(\alpha_1+\alpha_2)
$$
we get
\begin{equation}
\label{tiga2}
    7 \geq \; 3\phi+\psi-\alpha_1-\alpha_2+O\left(\frac{n}{\delta}\right).
\end{equation}

Let $\mathcal{D}$ denote  the set of layers with 2 colors, with singles on both side.  (Their cardinalities  added up to  $\alpha_1$.) Let $\mathcal{E}$ denote  the set of layers that are adjacent to at least one layer from $\mathcal{D}$. Hence all layers in $\mathcal{E}$ are singles. Let $\mathcal{F}$ denote  the set of remaining layers,
i.e. not in $\mathcal{D}\cup \mathcal{E}$. First note that
\begin{equation} \label{addtodiam}
|\mathcal{D}|+|\mathcal{E}|+| \mathcal{F}|=D+1.
\end{equation}
By the minimum degree condition, for all $i:0<i<D$ we have $\delta\leq \ell_{i-1}+\ell_{i}+\ell_{i+1}$. Hence,
$\delta |\mathcal{F}| \leq \sum_{i\in  \mathcal{F}} (  \ell_{i-1}+\ell_{i}+\ell_{i+1}   )\leq 3n(1-\alpha_1) +O(\delta)$ and
\begin{equation} \label{Fbound}
| \mathcal{F}|\leq     3(1-\alpha_1)\frac{n}{\delta}+O(1).
\end{equation}
It is not difficult to see that $|\mathcal{D}|+|\mathcal{E}|\leq 3s$. 
Using this observation with (\ref{addtodiam}) and (\ref{Fbound}), we obtain
$$3\phi=\frac{3\delta}{n}(|\mathcal{D}|+|\mathcal{E}|+| \mathcal{F}|-1)\leq \frac{3\delta}{n}(3s)+\frac{3\delta}{n}| \mathcal{F}|  +O\left( \frac{\delta}{n} \right)  $$
and hence 
\begin{equation}  \label{Laszlo_improvement}
\phi \leq 3\psi+ 3(1-\alpha_1)+ O\left( \frac{\delta}{n} \right)  .
\end{equation} 
We tried to use more inequalities and  more variables, splitting $\alpha_1$ further, based on the number of colors in the layers before and after. Removing 
redundant variables and conditions,  
we finalized our linear program  based on constraints  (\ref{obvious}, (\ref{huszonnyolcas}), (\ref{tiga2}) and (\ref{Laszlo_improvement}) as follows:\\
\begin{tabular}{lll}
\textbf{Maximize}& $\phi=\dfrac{D\delta}{n}$&\\
\textbf{subject to}\\
    &$\displaystyle{\color{white}12\phi+3\psi+4}\mu+{\color{white}2}\alpha_1+\alpha_2$ &$\displaystyle\leq 1$\\
    &$\displaystyle{\color{white}12\psi+}3\psi$&$\displaystyle\leq 2$\\    
    &$\displaystyle12 \phi {\color{white}3\psi+\,\,}+4\mu - 2\alpha_1 -\alpha_2$      &$\displaystyle\leq 28 + O\left(\frac{\delta}{n}\right)$\\
    &$\displaystyle {\color{white}1}3\phi{+ \color{white}3}\psi{\color{white}+4\mu+}-{\color{white}2}\alpha_1-\alpha_2$ &$\displaystyle\leq 7{\color{white}0}+O\left(\frac{\delta}{n}\right)$\\
    &$\displaystyle{\color{white}12}\phi-3\psi{\color{white}+4\mu\,\,}+3\alpha_1$    &$\displaystyle\leq 3{\color{white}0}+O\left(\frac{\delta}{n}\right)$\\
    &${\color{white}12121212121}\displaystyle \phi, \mu,  \psi, \alpha_1,  \alpha_2$    &$\displaystyle\geq 0 {\color{white}0 + O\left(\frac{\delta}{n}\right)}$
\end{tabular}

 Let 
$\mathbf{x}= (x_1,x_2,x_3,x_4,x_5)^T=    (\phi, \mu,  \psi, \alpha_1,  \alpha_2)^T$, let $A$ be the $5\times 5$ coefficient matrix above, $\mathbf{b}=\left(1,2,28,7, 3\right)^T$,  and $\mathbf{h}$ be any
concrete error term in the constraint column within the $O\left(\frac{\delta}{n}\right)$ bounds. Let $\mathbf{y}= (y_1,y_2,y_3,y_4,y_5)$.
Consider now   four closely related linear programs:
\begin{align}
A\mathbf{x}&\leq \mathbf{b}+\mathbf{h}; &\mathbf{x}\geq 0; &\hbox{\ \ maximize \ \ } x_1;   \label{eredeti}       \\
A\mathbf{x}&\leq \mathbf{b}; &\mathbf{x}\geq 0; &\hbox{\ \ maximize \ \ } x_1;   \label{eredetiveges}       \\
\mathbf{y}A&\geq (1,0,0,0,0,0,0); &\mathbf{y} \geq 0; &\hbox{\ \ minimize \ \ }   \mathbf{y}  (\mathbf{b}+\mathbf{h})^T ;  \label{eredetidual}       \\
\mathbf{y}A &\geq (1,0,0,0,0,0,0); &\mathbf{y}\geq 0; &\hbox{\ \ minimize \ \ }    \mathbf{y} \mathbf{b}^T  .  \label{eredetivegesdual}  
\end{align}
Our standard reference to linear programming is \cite{dantzig}. Note that (\ref{eredeti}) is identical to the displayed linear program, and that 
(\ref{eredeti})  and (\ref{eredetidual}), and (\ref{eredetiveges})  and (\ref{eredetivegesdual}) are dual linear programs, respectively, and the Duality Theorem
of Linear Programming applies to them.
Utilizing the open source online tool \cite{PHP}, we solved (\ref{eredetiveges}) 
with optimum $\phi=\frac{57}{23}$
attained at $( \frac{57}{23},0,\frac{13}{22},\frac{17}{23},\frac{6}{23}        )^T$. 
 By duality, $\frac{57}{23}$ is the optimum of (\ref{eredetivegesdual}) as well.
The polytope defined by the constraints of (\ref{eredetiveges}) 
has a feasible solution $\mathbf{x}^*$, for which  inequalities in the $3^{\text{rd}}, 4^{\text{th}}$ and $5^{\text{th}}$ constraints hold strictly---just modify the optimal solution by reducing $\phi$ a bit.  We want to show that (\ref{eredeti}) has a finite optimum, 
 if $n$ is sufficiently large. By the first constraint in (\ref{eredeti}), $\phi\leq 3$ for $n$  sufficiently large. Our only concern is whether (\ref{eredeti}) has a feasible 
 solution at all, as negative error terms might eliminate it. Clearly $\mathbf{x}^*$ is  a feasible 
 solution,  if $n$ is sufficiently large. By the Duality Theorem, (\ref{eredetidual}) has a finite minimum value, which is equal 
 to the maximum value for  (\ref{eredeti}). As the polytopes of (\ref{eredetidual}) and (\ref{eredetivegesdual})  are the same, this finite minimum is achieved in one 
 of the finitely many vertices of this polytope, say $\mathbf{y}^{(1)},...,\mathbf{y}^{(m)}$, as these linear programs only differ in their objective functions. Now we have
 \begin{eqnarray*}
\max \hbox{ \  $x_1$ in   (\ref{eredeti})  }&=&
 \min_{  \mathbf{y}\geq 0 } \mathbf{y} (\mathbf{b}+\mathbf{h})^T  =  \min_{ i=1}^m \mathbf{y}^{(i)}(\mathbf{b}+\mathbf{h})^T \\
& \geq &   \min_{ i=1}^m \mathbf{y}^{(i)}\mathbf{b}^T  + \min_{ i=1}^m \mathbf{y}^{(i)}\mathbf{h}^T =\frac{57}{23}+O\left(\frac{\delta}{n}\right).
 \end{eqnarray*}
On the other hand, 
 \begin{eqnarray*}
\max \hbox{ \  $x_1$ in   (\ref{eredeti})  }&=&  \min_{  \mathbf{y}\geq 0 }  \mathbf{y} (\mathbf{b}+\mathbf{h})^T =  \min_{ i=1}^m\Bigl(\mathbf{y}^{(i)}\mathbf{b}^T +\mathbf{y}^{(i)} \mathbf{h}^T\Bigl)\\
&\leq & \min_{ i=1}^m\Bigl(\mathbf{y}^{(i)} \mathbf{b}^T +\max_{i=1}^m \mathbf{y}^{(i)} \mathbf{h}^T\Bigl)
=  \min_{ i=1}^m   \mathbf{y}^{(i)} \mathbf{b}^T +\max_{i=1}^m \mathbf{y}^{(i)} \mathbf{h}^T\\
&=& \frac{57}{23}+O\left(\frac{\delta}{n}\right).
 \end{eqnarray*}
 We concluded the proof of Theorem \ref{th:189o76}.  The linear programming arguments above should be well-known, but we were unable to find a reference.

The following theorem proves     the weaker version of Conjecture~\ref{con:7o3} for $k=3$, in a restricted case of no single layers: 
\begin{theorem}
For every connected $3$-colorable graph $G$ of order $n$ and minimum degree at least
$\delta\ge 1$, such that in the canonical clump graph of $G$ no layer $L_i$ is a single for $0<i<D$, we have
$$ \diam(G) \leq \frac{7n}{3\delta}+O(1). $$
\label{special}
\end{theorem}
\begin{proof}
 If there are no single color layers besides $L_0$ and $L_D$, in (\ref{pairsummary'}) the second and third sums are zero, 
and the first is upper bounded by $\frac{2}{3}n$. 
This yields $14n/3\geq 2D\delta+O(\delta)$. An alternative proof of the theorem is from \ref{Evatriples2}, in which $s=0$ and the sum is $O(\delta)$ in this case. 
\end{proof}
The theorem also holds if the number of single layers is bounded as $n\rightarrow \infty$.
We are not aware of constructions getting close to this upper bound without single layers.

\end{document}